\documentclass[10pt]{amsart}
\usepackage{latexsym}
\usepackage{amsmath}
\usepackage{amssymb}
\usepackage{mathrsfs}
\usepackage{graphicx}
\usepackage{color}
\usepackage{pgfpages}
\usepackage{ifthen}
\usepackage{leftidx,tensor}
\usepackage[T1]{fontenc}
\usepackage[latin1]{inputenc}
\usepackage{mathtools}
\usepackage{comment}
\usepackage{dsfont}
\usepackage[nocompress]{cite}
\usepackage{stmaryrd}

\usepackage[shortlabels]{enumitem}
\usepackage{aliascnt}
\usepackage[bookmarks=true,pdfstartview=FitH, pdfborder={0 0 0}, colorlinks=true,citecolor=red, linkcolor=blue]{hyperref}
\usepackage{bbm}
\usepackage{nicefrac}

\theoremstyle{plain}
\newtheorem{thm}{Theorem}[section]
\newaliascnt{cor}{thm}
\newaliascnt{prop}{thm}
\newaliascnt{lem}{thm}
\newtheorem{cor}[cor]{Corollary}
\newtheorem{prop}[prop]{Proposition}
\newtheorem{lem}[lem]{Lemma}
\aliascntresetthe{cor}
\aliascntresetthe{prop}
\aliascntresetthe{lem} 
\theoremstyle{definition}
\newaliascnt{defn}{thm}
\newaliascnt{asu}{thm}
\newaliascnt{con}{thm}
\newtheorem{defn}[defn]{Definition}
\newtheorem{asu}[asu]{Assumption}

\aliascntresetthe{defn}
\aliascntresetthe{asu}
\aliascntresetthe{con}
\newcounter{stp}
\newcounter{stpi}
\newcounter{stpci}
\newcounter{stpiii}

\theoremstyle{remark}
\newaliascnt{rem}{thm}
\newaliascnt{exa}{thm}
\newaliascnt{masu}{thm}
\newaliascnt{nota}{thm}
\newaliascnt{sett}{thm}
\newtheorem{rem}[rem]{Remark}

\aliascntresetthe{rem}
\aliascntresetthe{exa}
\aliascntresetthe{masu}
\aliascntresetthe{nota}
\aliascntresetthe{sett}
\numberwithin{equation}{section}

\setlist[enumerate]{font = \normalfont}

\newcommand{\F}{\mathbb{F}}
\newcommand{\E}{\mathbb{E}}
\newcommand{\R}{\mathbb{R}}
\newcommand{\C}{\mathbb{C}}

\newcommand{\bP}{\mathbb{P}}

\newcommand{\D}{\mathbb{D}}

\newcommand{\bX}{\mathbb{X}}

\newcommand{\rW}{\mathrm{W}} 
\newcommand{\rL}{\mathrm{L}}

\newcommand{\rH}{\mathrm{H}}
\newcommand{\rT}{\mathrm{T}}
\newcommand{\rC}{\mathrm{C}}
\newcommand{\rB}{\mathrm{B}}

\newcommand{\rD}{\mathrm{D}}
\newcommand{\rR}{\mathrm{R}}
\newcommand{\rN}{\mathrm{N}}
\newcommand{\rX}{\mathrm{X}}
\newcommand{\rY}{\mathrm{Y}}

\newcommand{\rI}{\mathrm{I}}
\newcommand{\rG}{\mathrm{G}}

\newcommand{\DeltaD}{\Delta_{\mathrm{D}}}

\newcommand{\rd}{\,\mathrm{d}}

\newcommand{\fs}{\mathrm{fs}}

\newcommand{\mre}{\mathrm{e}}

\newcommand{\loc}{\mathrm{loc}}

\newcommand{\mS}{m_\mathcal{S}}
\newcommand{\Afs}{\cA_\fs}

\newcommand{\cS}{\mathcal{S}}
\newcommand{\cF}{\mathcal{F}}

\newcommand{\cO}{\mathcal{O}}
\newcommand{\cB}{\mathcal{B}}
\newcommand{\cL}{\mathcal{L}}

\newcommand{\cR}{\mathcal{R}}

\newcommand{\cH}{\mathcal{H}}
\newcommand{\cA}{\mathcal{A}}
\newcommand{\tcA}{\Tilde{\mathcal{A}}}
\newcommand{\cP}{\mathcal{P}}
\newcommand{\cC}{\mathcal{C}}

\newcommand{\Hinfty}{{\cH^\infty}}

\newcommand{\cSMR}{\mathcal S \mathcal M \mathcal R}

\newcommand{\frB}{\mathfrak{B}}
\newcommand{\frF}{\mathfrak{F}}
\newcommand{\frP}{\mathfrak{P}}

\newcommand{\tu}{\Tilde{u}}

\newcommand{\tQ}{\Tilde{Q}}
\newcommand{\tp}{\Tilde{p}}

\newcommand{\bfone}{\boldsymbol{1}}

\newcommand{\eps}{\varepsilon}
\newcommand{\del}{\partial}
\newcommand{\llb}{\llbracket}
\newcommand{\rrb}{\rrbracket}
\DeclareMathOperator{\tr}{tr}

\DeclareMathOperator{\Id}{Id}
\DeclareMathOperator{\mdiv}{div}

\newcommand{\tin}{\enspace \text{in} \enspace}
\newcommand{\ton}{\enspace \text{on} \enspace}
\newcommand{\tfor}{\enspace \text{for} \enspace}
\newcommand{\tforall}{\enspace \text{for all} \enspace}
\newcommand{\tand}{\enspace \text{and} \enspace}
\newcommand{\tso}{\enspace \text{so} \enspace}
\newcommand{\tif}{\enspace \text{if} \enspace}

\newcommand{\twhere}{\enspace \text{where} \enspace}
\newcommand{\twith}{\enspace \text{with} \enspace}

\begin{document}

\title[Analysis of Stochastic Fluid-Rigid Body Dynamics via Stochastic Maximal Regularity]{Strong Well-posedness for a Stochastic Fluid-Rigid Body System via Stochastic Maximal Regularity}

\author{Felix Brandt}
\address{Department of Mathematics, University of California at Berkeley, Berkeley, 94720, CA, USA.}
\email{fbrandt@berkeley.edu}
\author{Arnab Roy}
\address{Basque Center for Applied Mathematics (BCAM), Alameda de Mazarredo 14, 48009 Bilbao, Spain.}
	\address{IKERBASQUE, Basque Foundation for Science, Plaza Euskadi 5, 48009 Bilbao, Bizkaia, Spain.}
\email{aroy@bcamath.org}

\begin{abstract}
We develop a rigorous analytical framework for a coupled stochastic fluid-rigid body system in \(\mathbb{R}^3\). 
The model describes the motion of a rigid ball immersed in an incompressible Newtonian fluid subjected to both additive noise in the fluid and body equations and transport-type noise in the fluid equation. 
We establish local strong well-posedness of the resulting system by combining stochastic maximal \(\rL^p\)-regularity theory with a decoupling approach for the associated fluid-structure operator.
A key step is to prove the boundedness of the \(\mathcal{H}^\infty\)-calculus for this operator. 
In addition, we provide blow-up criteria for the maximal existence time of solutions.
To our knowledge, this is the first rigorous treatment of strong solutions of stochastic fluid-structure interactions.
\end{abstract}

\maketitle

\section{Introduction}

In this paper, we analyze a stochastic fluid-rigid body interaction problem of a viscous, incompressible Newtonian fluid with a rigid ball in 3D.
First, let us describe the underlying deterministic problem.
For a time $t \in [0,T]$, where $T > 0$, denote by $h(t)$ the position of the center of mass of the rigid ball centered around the origin at time $t = 0$, by $\cS(t)$ its domain, and by $\cF(t) = \R^3 \setminus \cS(t)$ the domain occupied by the fluid.
For the fluid velocity $u \colon [0,T] \times \cF(t) \to \R^3$ and pressure $p \colon [0,T] \times \cF(t) \to \R$, the linear and angular velocity $\ell \colon [0,T] \to \R^3$ and $\omega \colon [0,T] \to \R^3$, i.e., $h'(t) = \ell(t)$, the mass of the ball $\mS$, the time-independent inertia tensor $J_0$, a viscosity parameter $\nu > 0$, the unit vector field $n$ normal to $\del \cS(t)$ and directed towards the interior of $\cS(t)$ as well as the Cauchy stress tensor $\rT(u,p)$ that takes the shape
\begin{equation}\label{eq:Cauchy stress tensor}
    \rT(u,p) = 2 \nu(\nabla u + \nabla u^\top) - p \Id_3,
\end{equation}
the deterministic interaction problem is given by
\begin{equation}\label{eq:fsi problem with ball moving dom}
    \left\{
    \begin{aligned}
        \del_t u + (u \cdot \nabla)u - \nu \Delta u + \nabla p
        &= 0, \enspace \mdiv u = 0, &&\tfor t > 0, \enspace y \in \cF(t),\\
        \mS \ell'(t) + \int_{\del \cS(t)} \rT(u,p) n \rd \Gamma
        &= 0, &&\tfor t > 0,\\
        J_0 \omega'(t) + \int_{\del \cS(t)} (y - h(t)) \times \rT(u,p) n \rd \Gamma
        &= 0, &&\tfor t > 0,\\
        u(t,y)
        &= \ell(t) + \omega(t) \times (y-h(t)), &&\tfor t > 0, \enspace y \in \del \cS(t),\\
        u(0) = u_0, \enspace \ell(0)
        &= \ell_0, \enspace \omega(0) = \omega_0.
    \end{aligned}
    \right.
\end{equation}
As it does not affect the analysis, we will assume $\nu \equiv 1$ in the remainder of this paper.
Thanks to the simple geometry of the ball, the transformation of \eqref{eq:fsi problem with ball moving dom} to a fixed domain is straightforward.
Indeed, invoking the change of variables $x \mapsto y(t,x) \coloneqq x + h(t)$, we set
\begin{equation}\label{eq:change of var}
    v_0(x) \coloneqq u_0(x), \enspace v(t,x) \coloneqq u(t,x+h(t)) \tand \pi(t,x) \coloneqq p(t,x+h(t)).
\end{equation}
This transformation leads to the introduction of the term $-(\ell \cdot \nabla)v$ in the fluid equation and allows to reformulate \eqref{eq:fsi problem with ball moving dom} on the fixed domain denoted by $\cF_0 \coloneqq \cF(0)$.
Moreover, we set $\cS_0 \coloneqq \cS(0)$ and accordingly $\del \cS_0 \coloneqq \del \cS(0)$.
Thus, we get
\begin{equation}\label{eq:fsi problem with ball fixed dom}
    \left\{
    \begin{aligned}
        \del_t v + ((v-\ell) \cdot \nabla)v - \Delta v + \nabla \pi
        &= 0, \enspace \mdiv v = 0, &&\tfor t > 0, \enspace x \in \cF_0,\\
        \mS \ell'(t) + \int_{\del \cS_0} \rT(v,\pi) n \rd \Gamma
        &= 0, &&\tfor t > 0,\\
        J_0 \omega'(t) + \int_{\del \cS_0} x \times \rT(v,\pi) n \rd \Gamma
        &= 0, &&\tfor t > 0,\\
        v(t,y)
        &= \ell(t) + \omega(t) \times x, &&\tfor t > 0, \enspace x \in \del \cS_0,\\
        v(0) = v_0, \enspace \ell(0)
        &= \ell_0, \enspace \omega(0) = \omega_0.
    \end{aligned}
    \right.
\end{equation}
The investigation of {\em deterministic} fluid-rigid body interaction problems is a classical topic in mathematical fluid mechanics and has been studied abundantly.
At this stage, we refer to the pioneering work of Serre \cite{Ser:87}, the articles \cite{CSMT:00, DE:99} in the weak setting or the papers \cite{EMT:23, Galdi:02, GS:02, GGH:13, MT:18, Takahashi:03, TT:04, WX:11} in the strong setting.
For an overview of fluid-flow structure interaction problems, see also \cite{KKLTTW:18}.

The investigation of {\em stochastic} fluid-structure interaction (FSI) problems, where the structure is assumed to be elastic and located at a part of the fluid boundary, has only started recently.
The study of stochastic FSI problems is motivated by real world problems including the modeling of the blood flow in arteries.
These problems are naturally subject to noise, and well-posedness results of the corresponding stochastic interaction problems provide confidence that the models are robust under the addition of noise.
Kuan and \v{C}ani\'c \cite{KC:22} pioneered the investigation of stochastic FSI problems by analyzing a reduced model of a linearly coupled FSI problem consisting of a stochastic viscous wave equation.
The same authors considered in \cite{KC:24} the 2D/1D interaction problem of a Stokes fluid and a stochastically forced linearly elastic membrane modeled by a linear wave equation along a fixed interface, i.e., in the linearly coupled case.
These considerations were extended by Tawri and \v{C}ani\'c \cite{TC:25} to 2D Navier-Stokes equations nonlinearly coupled to an elastic lateral wall, where the fluid and the membrane equation are subject to stochastic forces.
Tawri investigated the case of non-zero longitudinal displacement \cite{Taw:25} and Navier slip boundary conditions \cite{Taw:24}.
Breit et al.\ \cite{BMM:24} studied the interaction of a 3D incompressible Navier-Stokes with a linearly elastic Koiter type shell subject to transport noise.
Let us emphasize that the aforementioned results all concern solutions in the {\em weak} PDE sense.

In this paper, for the first time, we analyze a stochastic fluid-rigid body interaction problem.
One application of such problems concerns polymer dynamics, see also \cite{WGM:10}. Recent advances in the modeling of fluctuating hydrodynamics and stochastic fluid-structure interactions have enabled simulations of rigid particle suspensions and thermally-driven dynamics in fluids \cite{sprinkle2017large, atzberger2011stochastic}.
In contrast to the theoretical works mentioned in the former paragraph, we focus here on solutions in the {\em strong} PDE sense.
For a family of standard independent Brownian motions $(W_t^n)_{n \ge 1}$ on a filtered probability space $(\Omega,\cA, \F = (\frF_t)_{t \ge 0},\bP)$ with associated expectation~$\E[\cdot]$, $b = (b_n)_{n \ge 1}$ and $f = (f_n)_{n \ge 1} = ((f_v^n,f_\ell^n,f_\omega^n)^\top)_{n \ge 1}$ made precise in \autoref{sec:main results}, and for the {\em random variables} $v \colon [0,T] \times \Omega \times \cF_0 \to \R^3$, $\pi \colon [0,T] \times \Omega \times \cF_0 \to \R$, $\ell \colon [0,T] \times \Omega \to \R^3$, and $\omega \colon [0,T] \times \Omega \to \R^3$, we consider the stochastic fluid-rigid body interaction problem
\begin{equation}\label{eq:stoch forced fsi problem intro}
    \left\{
    \begin{aligned}
        \rd v - \Delta v \rd t + \nabla \pi \rd t
        &= -((v - \ell) \cdot \nabla)v \rd t + \sum_{n \ge 1} \bigl[(b_n \cdot \nabla)v + f_{v}^n\bigr] \rd W_t^n, 
        &&\tin \cF_0,\\
        \mdiv v
        &= 0, &&\tin \cF_0,\\
        \rd \mS \ell + \int_{\del \cS_0} \rT(v,\pi) n \rd \Gamma \rd t
        &= \sum_{n \ge 1} f_{\ell}^n \rd W_t^n,\\
        \rd J_0 \omega + \int_{\del \cS_0} x \times \rT(v,\pi) n \rd \Gamma
        &= \sum_{n \ge 1} f_{\omega}^n \rd W_t^n,\\
        v
        &= \ell + \omega \times x, &&\ton \del \cS_0,\\
        v(0) = v_0, \enspace \ell(0)
        &= \ell_0, \enspace \omega(0) = \omega_0, &&\tin \cF_0.
    \end{aligned}
    \right.
\end{equation}
Note that we consider additive noise in the fluid and body velocity equations, while we also take into account transport noise in the fluid equation.
Transport noise has a clear physical meaning in fluid mechanical transport processes, see \cite{CGH:17}, and its role in the Navier-Stokes equations has been discussed extensively in the monograph \cite{FL:23}.
The role of transport noise for turbulent models such as the Navier-Stokes equations, the surface quasi-geostrophic equations, and the primitive equations has also been explored in the recent article \cite{DP:26}.
Fluid dynamics modeled by stochastic flow with the turbulent term driven by white noise was studied in \cite{MR:04}.
For some ill-posed PDEs, transport noise has a regularizing effect \cite{FGP:10}.
We observe that, at least formally, it is also possible to derive the stochastic interaction problem \eqref{eq:stoch forced fsi problem intro} directly from \eqref{eq:fsi problem with ball moving dom} upon transforming to the fixed domain.

The aim of this paper is to prove that there exists a unique, local, strong solution $(v,\ell,\omega)$, which is a random variable in the present setting, to \eqref{eq:stoch forced fsi problem intro} when considering strongly $\frF_0$-measurable initial data in suitable Besov spaces with compatibility conditions.
Furthermore, we discuss instantaneous regularization of the solution, establish a sort of continuous dependence of the solution on the initial data and provide blow-up criteria.
For a precise statement of the main results, see \autoref{sec:main results}.

Our approach relies on the theory of so-called {\em stochastic maximal regularity}.
In the {\em deterministic} case, maximal regularity of an operator $A \colon \rX_1 \subset \rX_0 \to \rX_0$ on a Banach space $\rX_0$ amounts to establishing that the parabolic operator
\begin{equation*}
    \Bigl(\frac{\rd}{\rd t} + A,\tr\Bigr) \colon \mathcal{E}_1 \to \mathcal{E}_0 \times \rX_\gamma,
\end{equation*}
for the time trace $\tr$, is an isomorphism between suitable Banach spaces $\mathcal{E}_1$ and $\mathcal{E}_0 \times \rX_\gamma$.
Thus, it allows to tackle quasilinear evolution equations of the form
\begin{equation*}
    u'(t) + A(u(t))u(t) = F(u(t)), \tfor t > 0, \enspace u(0) = u_0,
\end{equation*}
with nonlinear terms $A(\cdot)$ and $F(\cdot)$, by means of a linearization and fixed-point procedure.
This method has successfully been applied to a wide array of problems in applied analysis.
For more details on the theory as well as applications, we also refer to the monographs \cite{Ama:95, Lun:95, DHP:03, PS:16}.

{\em Stochastic} maximal regularity provides an analogue for the study of stochastic evolution equations 
\begin{equation*}
    \rd u + A(u)u \rd t = F(u) \rd t + (B(u)u + G(u)) \rd W, \tfor t > 0, \enspace u(0) = u_0,
\end{equation*}
for a Brownian motion $W$ and certain $A$, $B$, $F$ and $G$.
These quasilinear stochastic evolution equations are typically linearized by suitable operators $A$ and $B$, and stochastic maximal regularity can again be interpreted as an isomorphism between Banach spaces.
Here it is important to note that strong solutions are generally in spaces with lower regularity in time due to the roughness of the noise.

Remarkably, van Neerven et al.~\cite{vNVW:12a} revealed a sufficient {\em deterministic} condition for stochastic maximal regularity in the case $B = 0$, namely, they showed that under specific assumptions on the underlying Banach space, which are made precise in \autoref{ssec:stoch max reg and semilin SPDEs}, stochastic maximal regularity can be achieved if the operator $A$ admits a bounded $\Hinfty$-calculus.
For further details on the latter concept, we also refer to \cite{DHP:03, KW:04}.
The result in \cite{vNVW:12a} was later on extended to the situation of non-zero $B$ by Portal and Veraar \cite{PV:19}.
In a series of articles, Agresti and Veraar \cite{AV:22a, AV:22b} developed frameworks to nonlinear stochastic evolution equations in so-called {\em critical spaces}.
Here criticality refers to the presence of a scaling invariance.
The introduction of time weights allowed the authors to lower the regularity of initial data, exploit parabolic regularization, and establish blow-up criteria.
The aforementioned framework has been applied to many problems lately, and we refer here to the recent survey of Agresti and Veraar \cite{AV:25} and the references therein.
For critical spaces in the deterministic case, see \cite{PSW:18} or also the survey \cite{Wil:23}.
For different approaches to SPDEs, we also refer to the monographs \cite{DPZ:14} and \cite{FL:23}, which especially addresses stochastic fluid equations, and the references therein.

In order to analyze \eqref{eq:stoch forced fsi problem intro}, we reformulate the problem in operator form based on the so-called {\em fluid-structure operator}. 
In light of the sufficient condition for stochastic maximal regularity \cite{vNVW:12a}, it is a key step in our analysis to establish the boundedness of the $\Hinfty$-calculus of this operator. Note that the boundedness of the $\Hinfty$-calculus is a stronger property, and it is generally harder to verify, since it exhibits a worse behavior under perturbations, see also \cite{McIY:90}.
Another difficulty arising in the investigation of the fluid-structure operator is that it has a {\em non-diagonal domain}, meaning that the fluid-rigid body coupling is incorporated into the domain of the operator.
To bypass this issue, we employ a decoupling approach, i.e., we use a similarity transform to reduce the study to an operator that has diagonal domain, but that is of a more complicated shape.
By investigating the weak Neumann problem associated with the pressure with less regular data, and by using a perturbation argument for the bounded $\Hinfty$-calculus, we then manage to show that the fluid-structure operator admits a bounded $\Hinfty$-calculus.
This paves the way for the stochastic maximal regularity and also reveals interesting functional analytic properties such as a characterization of the fractional power domains of the fluid-structure operator.

The proof of the local strong well-posedness of \eqref{eq:stoch forced fsi problem intro} is completed by nonlinear estimates, which, in conjunction with the stochastic maximal regularity established in the previous step, allow us to use the theory of semilinear stochastic evolution equations. 
The nonlinear estimates require a characterization of the interpolation spaces, which is also carried out in this paper. 

To summarize, this paper makes several novel contributions to the analysis of stochastic fluid-structure interaction problems. 
First, we formulate and rigorously study a coupled stochastic system describing the interaction between a rigid body and an incompressible Newtonian fluid in
$\R^3$, incorporating both additive and transport-type stochastic forcing. 
A central result of the paper is the boundedness of the $\Hinfty$-calculus for the associated fluid-structure operator in appropriate weighted function spaces. 
This significantly extends existing results on analyticity and maximal $\rL^p$-regularity for such systems, particularly in unbounded domains. 
Our analysis also includes a characterization of the relevant interpolation spaces and a blow-up criterion for the maximal existence time of solutions.
While in the Hilbert space case, our initial data regularity is consistent with known deterministic results, we also indicate how the theory of critical spaces permits a further relaxation of initial regularity in the case of a rigid ball (cf.~\autoref{rem:crit of init data}), analogous to the approach in \cite[Section~3]{PW:17}. 
To our knowledge, this work provides \textit{the first rigorous treatment of strong solutions for a fluid-rigid body system under stochastic forcing}.

We emphasize that our results extend and substantially strengthen several earlier findings on the deterministic fluid-structure operator. 
In the Hilbert space setting, Takahashi and Tucsnak \cite{Takahashi:03, TT:04} established analyticity of the associated semigroup. 
Wang and Xin \cite{WX:11} extended the analyticity of semigroup to the space $\rL^{\nicefrac{6}{5}}(\R^3) \cap \rL^q(\R^3)$ for $q \ge 2$. 
Further, $\rL^p$-theory for fluid-rigid body interaction problems of Newtonian and non-Newtonian fluids were obtained by Geissert et al. \cite{GGH:13}. 
In bounded domains, Maity and Tucsnak \cite{MT:18} employed $\cR$-boundedness techniques to establish maximal $\rL^p$-regularity and exponential stability of the fluid-structure operator, while recent work by Ervedoza et al.~\cite{EMT:23} addresses analyticity on subspaces of $\rL^q$ for general $q \in (1,\infty)$ in exterior domains. 
We point out that our results, notably, the boundedness of the $\Hinfty$-calculus of the fluid-structure operator (see \autoref{thm:bdd H00-calculus of fluid-structure op incompr Newtonian}, \autoref{cor:cons of H00-calc of fs op incompr Newtonian} and \autoref{thm:props fluid-structure op bdd dom}) not only generalize these findings but in fact imply them (up to bounded analytic semigroups or spectral shifts in exterior domains). 

This article is organized as follows.
In \autoref{sec:Hoo, SMR & bilin SPDEs}, we provide probabilistic and deterministic concepts such as the bounded $\Hinfty$-calculus, stochastic maximal regularity, and a framework tailored to bilinear SPDEs required throughout the paper.
\autoref{sec:main results} is dedicated to the presentation of the main results of this paper on the local strong well-posedness of the stochastic interaction problem~\eqref{eq:stoch forced fsi problem intro} as well as the blow-up criteria.
The purpose of \autoref{sec:bdd Hoo-calculus fluid-structure operator} is to establish the bounded $\Hinfty$-calculus of the fluid-structure operator and to deduce the stochastic maximal regularity from there.
In \autoref{sec:proof main results}, we prove the main results stated in \autoref{sec:main results} by characterizing the interpolation spaces and showing suitable nonlinear estimates.
We conclude the paper in \autoref{sec:concl rems & further discussion} by giving an overview of possible extensions of the results in the present paper.

\section{Bounded $\Hinfty$-calculus, stochastic maximal regularity and bilinear SPDEs}\label{sec:Hoo, SMR & bilin SPDEs}

In this section, we recall some facts about the bounded $\Hinfty$-calculus, discuss its relation with the concept of so-called {\em stochastic maximal $\rL^p$-regularity} and provide a framework for bilinear SPDEs, including local existence and uniqueness, local continuity as well as blow-up criteria.
\autoref{ssec:prelims} presents (mainly) probabilistic preliminaries, \autoref{ssec:Hinfty block op and perturbation} discusses the bounded $\Hinfty$-calculus for block operator matrices as well as perturbation theory, and in \autoref{ssec:stoch max reg and semilin SPDEs}, we invoke the concept of stochastic maximal regularity and provide a framework tailored to bilinear SPDEs.

\newpage

\subsection{Preliminaries}\label{ssec:prelims}
\

In this section, $\rX$ denotes a Banach space.
The notion of a UMD space, representing {\em unconditional martingale differences}, will be important.
Even though it is a probabilistic concept, it admits a characterization in terms of the boundedness of the Hilbert transform.
More precisely, a Banach space $\rX$ is a {\em UMD space} if and only if the Hilbert transform is bounded on $\rL^p(\R;\rX)$ for some $p \in (1,\infty)$.
Let us observe that many classical reflexive spaces such as Lebesgue spaces $\rL^p$, (fractional) Sobolev spaces $\rW^{s,p}$, Bessel potential spaces~$\rH^{s,p}$, and Besov spaces $\rB_{pq}^s$ enjoy this property for $p$, $q \in (1,\infty)$.

Another concept in this context is the notion of a Banach space of (Rademacher) type~$2$, see also \cite[Chapter~7]{HvNVW:17}.
Again, we observe that the aforementioned spaces possess this property provided the parameters $p$ and $q$ satisfy $p$, $q \ge 2$.
In particular, $\rL^p$-spaces with $p \in (1,\infty)$ have type $p \wedge 2$.

With regard to stochastic maximal regularity, some probabilistic notation is required.
For this purpose, denote by $(\Omega,\cA, \F = (\frF_t)_{t \ge 0},\bP)$ a filtered probability space with associated expectation $\E[\cdot]$.
Let us start by making precise the notion of a cylindrical Brownian motion.
For a separable Hilbert space~$\rH$, we call a bounded and linear operator $W_\rH \colon \rL^2(\R_+;\rH) \to \rL^2(\Omega)$ a {\em cylindrical Brownian motion} in $\rH$ provided the random variable $W_\rH(f)$ satisfies the following properties:
\begin{enumerate}[(a)]
    \item $W_\rH(f)$ is centered Gaussian for every $f \in \rL^2(\R_+;\rH)$,
    \item $W_\rH(f)$ is $\frF_t$-measurable for each $t \in \R_+$ and $f \in \rL^2(\R_+;\rH)$ with support in $[0,t]$,
    \item $W_\rH(f)$ is independent of $\frF_t$ for all $t \in \R_+$ and $f \in \rL^2(\R_+;\rH)$ with support in $[t,\infty]$, and
    \item $\E(W_\rH(f_1) W_\rH(f_2)) = (f_1,f_2)_{\rL^2(\R_+;\rH)}$ for all $f_1$, $f_2 \in \rL^2(\R_+;\rH)$.
\end{enumerate}
The process $(W_\rH(t) h)_{t \ge 0}$ defined by $W_\rH(t)h \coloneqq W_\rH(\bfone_{(0,t]} h)$, for $h \in \rH$, then is a Brownian motion.

Concerning stochastic integration, we now invoke {\em $\gamma$-radonifying operators}.
Often, the space on which the noise is modeled is $\ell^2$.
Here we take into account the case of a general Hilbert space~$\rH$.
Consider a sequence of independent standard normal random variables $(\gamma_i)_{i \ge 1}$ on a probability space $(\Omega,\bP)$, and let~$(h_i)_{i \ge 1}$ be an orthonormal basis for $\rH$.
A bounded linear operator $T \colon \rH \to \rX$ belongs to $\gamma(\rH,\rX)$ if the series $\sum_{i=1}^\infty \gamma_i T h_i$ converges in $\rL^2(\Omega;\rX)$, and the $\gamma$-radonifying norm is defined by
\begin{equation*}
    \| T \|_{\gamma(\rH,\rX)} \coloneqq \left\| \sum_{i=1}^\infty \gamma_i T h_i \right\|_{\rL^2(\Omega;\rX)}.
\end{equation*}
In specific situations, the space $\gamma(\rH,\rX)$ admits a precise characterization.
Indeed, if $\rX$ is a Hilbert space, then $\gamma(\rH,\rX)$ is the space of Hilbert-Schmidt operators $\cL_2(\rH,\rX)$, while if $\rX = \rL^q(S)$, $q \in [1,\infty)$, for a measure space $(S,\Sigma,\mu)$, then $\gamma(\rH,\rX) = \rL^q(S;\rH)$.
This generalizes to $\rX = \rH^{s,q}(S)$ for $s \in \R$ and~$q \in (1,\infty)$, i.e., $\gamma(\rH,\rX) = \rH^{s,q}(S;\rH)$ in this case.
We also refer to \cite[Chapter~9]{HvNVW:17} for more details.

We say that a process $\phi \colon [0,T] \times \Omega \to \rX$ is {\em strongly progressively measurable}, or also simply {\em progressively measurable} if for all $t \in [0,T]$, it holds that $\left. \phi \right|_{[0,T]}$ is strongly $\frB([0,t]) \otimes \frF_t$-measurable, where $\frB$ represents the Borel $\sigma$-algebra.
Moreover, we denote by $\rL_{\frP}^p((0,T) \times \Omega;\gamma(\rH,\rX))$ the progressive measurable subspace of $\rL^p((0,T) \times \Omega;\gamma(\rH,\rX))$.
If $\rX$ is a UMD Banach space with type $2$, for every $p \in [0,\infty)$, it can be shown that the map $G \mapsto \int_0^{\cdot} G \rd W_\rH$ extends to a continuous linear operator from $\rL_{\frP}^p((0,T) \times \Omega;\gamma(\rH,\rX))$ into $\rL^p(\Omega;\rC([0,T];\rX))$, see \cite[Theorem~4.7]{vNVW:15a}.
For $p \in (0,\infty)$, there exists a constant $C = C(p,\rX)$ such that
\begin{equation*}
    \E \sup_{0 \le t \le T} \left\| \int_0^t G(s) \rd W_\rH(s) \right\|_\rX^p \le C \cdot \E \| G \|_{\rL^2(0,T;\gamma(\rH,\rX))}^p, \tforall G \in \rL_{\frP}^p((0,T) \times \Omega;\gamma(\rH,\rX)).
\end{equation*}
If $\rH$ is separable with orthonormal basis $(h_i)_{i \ge 1}$, then for all $p \in [0,\infty)$ and $G \in \rL_{\frP}^p((0,T) \times \Omega;\gamma(\rH,\rX))$, we have the series representation
\begin{equation*}
    \int_0^{\cdot} G(s) \rd W_\rH(s) = \sum_{i \ge 1} \int_0^{\cdot} G(s) h_i \rd W_\rH(s) h_i,
\end{equation*}
with convergence in $\rL^p(\Omega;\rC([0,T];\rX))$.

For the upcoming theory, we briefly recall the concept of stopping times.
A {\em stopping time} is a measurable map $\tau \colon \Omega \to [0,T]$ so that $\{\tau \le t\} \in \frF_t$ is valid for all $t \in [0,T]$.
We then use $\llb 0,\sigma \rrb$ to denote the stochastic interval $\llb 0,\sigma \rrb \coloneqq \{(t,\omega) \in [0,T] \times \Omega : 0 \le t \le \sigma(\omega)\}$.
Given $A \subset \Omega$ and stopping times $\tau$ and~$\mu$ such that $\tau \le \mu$, we define the set $[\tau,\mu] \times A \subset [0,T] \times \Omega$ by
\begin{equation*}
    [\tau,\mu] \times A \coloneqq \{(t,\omega) \in [0,T] \times A : \tau(\omega) \le t \le \mu(\omega)\}.
\end{equation*}
For $A \in \cA$, a map $u \colon A \times [0,\mu] \to \rX$ is referred to as {\em strongly progressively measurable} if the process
\begin{equation*}
    \bfone_{A \times [0,\mu]} u \coloneqq 
    \left\{
    \begin{aligned}
        u, &\ton A \times [0,\mu],\\
        0, &\enspace \text{else},
    \end{aligned}
    \right.
\end{equation*}
is strongly progressively measurable.
In the context of stopping times, we define the $\sigma$-algebra of the $\tau$-past by $\frF_\tau \coloneqq \bigl\{A \in \cA : \{\tau \le t\} \cap A \in \frF_t \tforall t \in [0,T]\bigr\}$.
Measurability questions in this context have been addressed in \cite[Lemma~2.15]{AV:22a}.
This paves the way for the following definition.

\begin{defn}
Let $T > 0$, $p \in \{0\} \cup [1,\infty)$, and consider a stopping time $\tau \colon \Omega \to [0,T]$ and a family of function spaces $(\rY_t)_{t \in [0,T]}$ such that for each $f \in \rY_T$ and every $t \in [0,T]$, it holds that $\left. f \right|_{[0,t]} \in \rY_t$ and $t \mapsto \| \left. f\right|_{[0,t]} \|_{\rY_t}$ is increasing.
By $u \in \rL_{\frP}^p(\Omega;\rY_\tau)$, we mean that there is a strongly progressively measurable $\Tilde{u} \in \rL^p(\Omega;\rY_T)$ with $\left.\Tilde{u}\right|_{\llb 0,\tau \rrb} = u$, and if $p \in [1,\infty)$, we define
\begin{equation*}
    \| u \|_{\rL^p(\Omega;\rY_\tau)} \coloneqq \E\left(\| \left.\Tilde{u}\right|_{[0,\tau]} \|_{\rY_\tau}^p\right).
\end{equation*}
\end{defn}

We conclude this subsection by recalling the concept of time weights.
For this purpose, consider $p \in (1,\infty)$ and $\kappa \in (-1,p-1)$.
We define $w_\kappa(t) \coloneqq |t|^\kappa$ for $t \in \R$.
By $\rL^p(I,w_\kappa;\rX)$, we then denote the space of strongly measurable functions $f \colon I \to \rX$ such that
\begin{equation*}
    \| f \|_{\rL^p(I,w_\kappa;\rX)}^p \coloneqq \int_I \| f(t) \|_{\rX}^p w_\kappa(t) \rd t < \infty.
\end{equation*}
In the unweighted case $\kappa = 0$, i.e., $w_\kappa = 1$, we also write $\rL^p(I;\rX)$ instead of $\rL^p(I,w_0;\rX)$.
The weighted Sobolev spaces $\rW^{m.p}(I,w_\kappa;\rX)$ are defined analogously, and for $\theta \in (0,1)$, the weighted Bessel potential spaces $\rH^{\theta,p}(I,w_\kappa;\rX)$ can be obtained by complex interpolation.
In the above context, for the particular case of $\rY_t = \rL^p(I_t,w_\kappa;\rX)$, we define $\rL_{\frP}^p(I_\tau \times \Omega,w_\kappa;\rX) \coloneqq \rL_{\frP}^p(\Omega;\rY_\tau)$.

\subsection{Bounded $\Hinfty$-calculus of block operator matrices and perturbation theory}\label{ssec:Hinfty block op and perturbation}
\

In the sequel, we discuss the notion of the boundedness of the $\cH^\infty$-calculus as well as its application to block operator matrices.
By $\Hinfty(\rX)$, we denote the class of operators admitting a bounded $\Hinfty$-calculus, and for $A \in \Hinfty(\rX)$, we use $\phi_{A}^\infty$ for the associated $\Hinfty$-angle.

In the following lemma, we collect some important implications of the boundedness of the $\Hinfty$-calculus.
The maximal $\rL^p$-regularity is implied by the characterization of maximal $\rL^p$-regularity in terms of $\cR$-boundedness of the resolvents in UMD spaces due to Weis \cite[Theorem~4.2]{Wei:01}, while the characterization of the fractional power domains in~(b) is a consequence of the boundedness of the imaginary powers, see, e.g., \cite[Theorem~2.5]{DHP:03}. 

\begin{lem}\label{lem:rel of Hinfty with other concepts}
Suppose that $A \in \Hinfty(\rX)$ with $\Hinfty$-angle $\phi_A^\infty \in [0,\pi)$.
\begin{enumerate}[(a)]
    \item If $\phi_A^\infty < \nicefrac{\pi}{2}$, then $-A$ generates a bounded analytic semigroup $(\mre^{-At})_{t \ge 0}$ of angle less than or equal to $\nicefrac{\pi}{2} - \phi_A^\infty$.
    If additionally, the space $\rX$ is a UMD Banach space, then $A$ enjoys maximal $\rL^p$-regularity on $\rX$.
    \item For $\alpha \in (0,1)$ and $\rX_\alpha \coloneqq (\rD(A^\alpha),\| \cdot \|_\alpha)$ with $\| x \|_\alpha \coloneqq \| x \| + \| A^\alpha x \|$, for $\rX_A$ denoting $\rD(A)$ endowed with the graph norm, and for $[\cdot,\cdot]_\alpha$ representing the complex interpolation functor, it is valid that $\rX_\alpha \cong [\rX,\rX_A]_\alpha$.
\end{enumerate}
\end{lem}

Next, we make precise the notion of a diagonally dominant block operator matrix.

\begin{defn}\label{def:diag dom op matrix}
Consider Banach spaces $\rX_1$ and $\rX_2$ as well as linear operators $A \colon \rD(A) \subset \rX_1 \to \rX_1$, $B \colon \rD(B) \subset \rX_2 \to \rX_1$, $C \colon \rD(C) \subset \rX_1 \to \rX_2$ and $D \colon \rD(D) \subset \rX_2 \to \rX_2$. 
Moreover, set $\rX \coloneqq \rX_1 \times \rX_2$.
If 
\begin{enumerate}[(a)]
    \item the operators $A$ and $D$ are closed, linear and densely defined, and
    \item $C$ is relatively $A$-bounded, and $B$ is relatively $D$-bounded, i.e., $\rD(D) \subset \rD(B)$, $\rD(A) \subset \rD(C)$, and there are constants $c_A$, $c_D$, $L \ge 0$ so that for all $x \in \rD(A)$ and $y \in \rD(D)$, we have
    \begin{equation*}
        \| C x \|_{\rX_2} \le c_A \cdot \| A x \|_{\rX_1} + L \cdot \| x \|_{\rX_1} \tand \| B y \|_{\rX_1} \le c_D \cdot \| D y \|_{\rX_2} + L \cdot \| y \|_{\rX_2},
    \end{equation*}
\end{enumerate}
then the operator matrix $\cA \colon \rD(\cA) \coloneqq \rD(A) \times \rD(D) \subset \rX \to \rX$ defined by
\begin{equation}\label{eq:block op matrix}
    \cA \binom{x}{y} \coloneqq \begin{pmatrix}
        A & B\\
        C & D
    \end{pmatrix} \binom{x}{y}, \tfor \binom{x}{y} \in \rD(\cA),
\end{equation}
is called diagonally dominant.
\end{defn}

As an auxiliary result, we invoke the following lemma on the preservation of the bounded $\cH^\infty$-calculus under similarity transforms.
The assertion can be found in \cite[Proposition~2.11(vi)]{DHP:03}.

\begin{lem}\label{lem:preservation of results under sim trafe}
Consider Banach spaces $\rX$ and $\rY$, $S \in \cL(\rX,\rY)$ bijective, and a closed linear operator $A$ on~$\rX$, and set $A_1 \coloneqq S A S^{-1}$. 
Then $A \in \cH^\infty(\rX)$ iff $A_1 \in \cH^\infty(\rY)$, and $\phi_A^\infty = \phi_{A_1}^\infty$.
\end{lem}

The next result follows directly from \cite[Cor.~7.2]{AH:23} upon invoking the similarity transform given by $S = S^{-1} = \begin{pmatrix} 0 & \Id\\ \Id & 0 \end{pmatrix}$ and employing \autoref{lem:preservation of results under sim trafe}.

\begin{prop}\label{prop:op theoret props of block op matrices}
Let $\cA$ as defined in \eqref{eq:block op matrix} be diagonally dominant in the sense of \autoref{def:diag dom op matrix}, and assume in addition that $D \in \cL(\rX_2)$ as well as $A \in \cH^\infty(\rX_1)$ with $\phi_A^\infty \in [0,\pi)$.
Then for each $\phi \in (\phi_A^\infty,\pi)$, there is $\lambda \ge 0$ with $\lambda + \cA \in \cH^\infty(\rX)$ and $\phi_{\lambda + \cA}^\infty \le \phi$.
\end{prop}

Let us observe that the diagonal dominance of $\cA$ together with the boundedness of $D$ already implies that $B \colon \rD(D) = \rX_2 \to \rX_1$ has to be bounded.

Next, we discuss perturbation theory for the boundedness of the $\Hinfty$-calculus.
Note that it is generally {\em not} preserved under relatively bounded perturbations, see \cite{McIY:90}.
The following perturbation result is based on the relative boundedness with respect to a fractional power, see \cite[Corollary~3.3.15]{PS:16}.

\begin{prop}\label{prop:pert of Hinfty-calculus}
Let $A \in \Hinfty(\rX)$, consider a linear operator $B$ on $\rX$ such that $\rD(A^\alpha) \subset \rD(B)$ for some $\alpha \in [0,1)$, and assume that there exist $a$, $b \ge 0$ with
\begin{equation*}
    \| B x \| \le a \cdot \| x \| + b \cdot \| A^\alpha x \|, \tfor x \in \rD(A^\alpha).
\end{equation*}
If additionally, $A + B$ is sectorial and invertible, then $A + B \in \Hinfty(\rX)$ with $\phi_{A+B}^\infty \le \max\{\phi_A^\infty,\phi_{A+B}\}$.
\end{prop}

\subsection{Stochastic maximal regularity and bilinear SPDEs}\label{ssec:stoch max reg and semilin SPDEs}
\ 

This subsection is dedicated to recalling the notion of stochastic maximal regularity and providing a framework for bilinear SPDEs tailored to our application to the fluid-rigid body interaction problem.

\noindent {\bf Stochastic maximal regularity and role of $\Hinfty(\rX_0)$.}

\noindent 
First, let us make some assumptions on the Banach spaces to be considered in the sequel.
In fact, we suppose that $\rX_0$ and $\rX_1$ are UMD Banach spaces with type $2$ such that $\rX_1 \hookrightarrow \rX_0$.
In the sequel, we will denote by $\rX_\theta$ the complex interpolation spaces, i.e., $\rX_\theta = [\rX_0,\rX_1]_{\theta}$ for $\theta \in (0,1)$, and for $\theta \in (0,1)$ as well as $p \in (1,\infty)$, we use $(\rX_0,\rX_1)_{\theta,p}$ to denote the real interpolation spaces.

For $A$, $B$, $f$ and $g$ made precise below, $a \in [0,\infty)$ and a stopping time $\tau$, we investigate
\begin{equation}\label{eq:lin SPDE}
    \left\{
    \begin{aligned}
        \rd u + A u \rd t
        &= f \rd t + (B u + g) \rd W_\rH, \tfor t \in [a,\tau],\\
        u(a)
        &= u_a.
    \end{aligned}
    \right.
\end{equation}
In the following, we will assume that $A \in \cL(\rX_1,\rX_0)$ and $B \in \cL(\rX_1,\gamma(\rH,\rX_{\nicefrac{1}{2}}))$, where we recall the latter space from \autoref{ssec:prelims}.
Before proceeding with the definition of stochastic maximal $\rL^p$-regularity, we clarify the notion of a strong solution to \eqref{eq:lin SPDE}.

\begin{defn}
Assume that $f \in \rL^1(0,\tau;\rX_0)$ a.s.\ and $g \in \rL^2(0,\tau;\gamma(\rH,\rX_{\nicefrac{1}{2}}))$ a.s.\ are both progressively measurable.
We say that a progressively measurable process $u \colon [0,\tau] \times \Omega \to \rX_0$ is a strong solution to~\eqref{eq:lin SPDE} on $[0,\tau]$ provided $u \in \rC([0,\tau];\rX_0) \cap \rL^2(0,\tau;\rX_1)$ a.s., and for each $t \in [0,\tau]$, a.s.\ it holds that
\begin{equation}\label{eq:int eq strong sol to lin SPDE}
    u(t) - u_0 + \int_0^t A u(s) \rd s = \int_0^t f(s) \rd s + \int_0^t (B u(s) + g(s)) \rd W_\rH(s).
\end{equation}
\end{defn}

The discussion from \autoref{ssec:prelims} reveals that the stochastic integral in \eqref{eq:int eq strong sol to lin SPDE} is well-defined.
We are now in the position to introduce the concept of stochastic maximal regularity.
Part~(a) of the definition concerns the spatial regularization, while (bi) and (bii) also capture the regularization in time, which is more delicate in the stochastic setting due to the roughness of the noise.

\begin{defn}\label{def:stoch max reg}
Consider $p \in [2,\infty)$ and $\kappa \in [0,\nicefrac{p}{2}-1) \cup \{0\}$.
\begin{enumerate}[(a)]
    \item For the pair $(A,B)$, we write $(A,B) \in \cSMR_{p,\kappa}$, provided for all $T \in (0,\infty)$, there is a constant $C_T$ such that for all $a \in [0,T]$ and every stopping time $\tau \colon \Omega \to [a,T]$, each progressively measurable $f \in \rL^p(\Omega;\rL^p(a,\tau;w_\kappa;\rX_0))$ and $g \in \rL^p(\Omega;\rL^p(a,\tau,w_\kappa;\gamma(\rH,\rX_{\nicefrac{1}{2}}))$, there exists a unique solution~$u$ to \eqref{eq:lin SPDE} on $[a,\tau]$ with $u_a = 0$ such that $u \in \rL^p(\Omega;\rL^p(a,\tau,w_\kappa;\rX_1))$, with estimate
    \begin{equation*}
        \| u \|_{\rL^p(\Omega;\rL^p(a,\tau,w_\kappa;\rX_1))} \le C_T\bigl(\| f \|_{\rL^p(\Omega;\rL^p(a,\tau,w_\kappa;\rX_0))} + \| g \|_{\rL^p(\Omega;\rL^p(a,\tau,w_\kappa;\gamma(\rH,\rX_{\nicefrac{1}{2}})))}\bigr).
    \end{equation*}
    \item[(bi)]  In the case $p \in (2,\infty)$ and $\kappa \in [0,\nicefrac{p}{2}-1)$, we write $(A,B) \in \cSMR_{p,\kappa}^\bullet$ if $(A,B) \in \cSMR_{p,\kappa}$, and for each $\theta \in (0,\nicefrac{1}{2})$ and $T \in (0,\infty)$, there is a constant $C_{T,\theta}$ so that for all $[a,b] \subset [0,T]$, for every progressively measurable $f \in \rL^p(\Omega;\rL^p(a,b;w_\kappa;\rX_0))$ and $g \in \rL^p(\Omega;\rL^p(a,b,w_\kappa;\gamma(\rH,\rX_{\nicefrac{1}{2}}))$, the strong solution $u$ to \eqref{eq:lin SPDE} on $[a,b]$ with $u_a = 0$ fulfills the estimate
    \begin{equation*}
        \| u \|_{\rL^p(\Omega;\rH^{\theta,p}(a,b,w_\kappa;\rX_{1-\theta}))} \le C_{T,\theta}\bigl(\| f \|_{\rL^p(\Omega;\rL^p(a,b,w_\kappa;\rX_0))} + \| g \|_{\rL^p(\Omega;\rL^p(a,b,w_\kappa;\gamma(\rH,\rX_{\nicefrac{1}{2}})))}\bigr).
    \end{equation*}
    \item[(bii)] If $p = 2$ and accordingly $\kappa = 0$, we write $(A,B) \in \cSMR_{2,0}^\bullet$ if $(A,B) \in \cSMR_{2,0}$, and for each~$T \in (0,\infty)$, there is a constant $C_T$ so that for all $[a,b] \subset [0,T]$, for each progressively measurable $f \in \rL^2(\Omega;\rL^2(a,b;\rX_0))$ and $g \in \rL^2(\Omega;\rL^2(a,b;\gamma(\rH,\rX_{\nicefrac{1}{2}}))$, the solution $u$ to \eqref{eq:lin SPDE} on~$[a,b]$ with $u_a = 0$ satisfies
    \begin{equation*}
        \| u \|_{\rL^2(\Omega;\rC([a,b];\rX_{\nicefrac{1}{2}}))} \le C_T\bigl(\| f \|_{\rL^2(\Omega;\rL^2(a,b;\rX_0))} + \| g \|_{\rL^2(\Omega;\rL^2(a,b;\gamma(\rH,\rX_{\nicefrac{1}{2}})))}\bigr).
    \end{equation*}
\end{enumerate}
If $(A,B) \in \cSMR_{p,\kappa}^\bullet$, we say that $(A,B)$ admits stochastic maximal $\rL_\kappa^p$-regularity.
\end{defn}

The case of non-zero initial values can also be handled, see \cite[Section~3.3]{AV:25} for more details.
Below, we provide a sufficient condition for stochastic maximal $\rL_\kappa^p$-regularity in the case $B = 0$ in terms of the boundedness of the $\Hinfty$-calculus.
In the unweighted case, it is due to van Neerven, Veraar and Weis \cite{vNVW:12a, vNVW:12b}, while the extension to the weighted setting was achieved by Agresti and Veraar \cite[Section~7]{AV:20}.

\begin{prop}\label{prop:stoch max reg via Hinfty subset of Lq}
Consider a $\sigma$-finite measure space $(\cO,\Sigma,\mu)$, and let $q \in [2,\infty)$.
Moreover, assume that $\rX_0 = \rL^q(\cO)$, or that $\rX_0$ is isomorphic to a closed subspace of $\rL^q(\cO)$.
If there is $\lambda \ge 0$ such that~$\lambda + A \in \Hinfty(\rX_0)$ with $\phi_{\lambda + A}^\infty < \nicefrac{\pi}{2}$, then $(A,0) \in \cSMR_{p,\kappa}^\bullet$ for all $p \in (2,\infty)$ and $\kappa \in [0,\nicefrac{p}{2}-1)$.
In the case $q = 2$, we also get $(A,0) \in \cSMR_{2,0}^\bullet$.
\end{prop}

Finally, in order to allow for (small) transport noise, we also include the following perturbation result, which follows from \cite[Theorem~3.2]{AV:24}.

\begin{lem}\label{lem:pert SMR}
Suppose that $(A,B) \in \cSMR_{p,\kappa}^\bullet \cap \cSMR_p$, and fix $\delta \in (\nicefrac{(1+\kappa)}{p},\nicefrac{1}{2})$ if $p > 2$, and any $\delta \in (0,\nicefrac{1}{2})$ if $p=2$.
Moreover, let $B_0 \in \cL(\rX_1,\gamma(\rH,\rX_{\nicefrac{1}{2}}))$ such that for $C_B \in (0,1)$ sufficiently small, there exists $L_B \in (0,\infty)$ such that for all $x \in \rX_1$ and $t \in (0,T)$, we have
\begin{equation*}
    \| B_0 x \|_{\gamma(\rH,\rX_{\nicefrac{1}{2}})} \le C_B \cdot \| x \|_{\rX_1} + L_B \cdot \| x \|_{\rX_0}.
\end{equation*}
Then we obtain $(A,B + B_0) \in \cSMR_{p,\kappa}^\bullet$.
\end{lem}

\noindent {\bf Bilinear SPDEs.}

\noindent
After recalling the concept of stochastic maximal regularity, we next discuss its application to bilinear SPDEs.
More precisely, the stochastic evolution equations under consideration take the form
\begin{equation}\label{eq:bilin SPDE}
    \left\{
    \begin{aligned}
        \rd u + A u \rd t
        &= F(u) \rd t + (B u + g) \rd W_\rH,\\
        u(0)
        &= u_0,
    \end{aligned}
    \right.
\end{equation}
where we further suppose that the nonlinear term $F$ is bilinear, i.e., $F(u) = G(u,u)$ for some suitable~$G$.
In the following, we make precise the assumptions on $A$, $B$, $F$ and $g$.
Let us start by saying that for the leading operators $A$ and $B$, we assume that the pair $(A,B)$ admits maximal stochastic $\rL^p$-regularity in the sense of \autoref{def:stoch max reg}(b).
As a result, for $w_\kappa(t) = t^\kappa$ and $\kappa \in [0,\nicefrac{p}{2}-1)$, the natural path space for the solution of \eqref{eq:bilin SPDE} in the case $p > 2$ is
\begin{equation*}
    \bigcap_{\theta \in [0,\nicefrac{1}{2})} \rH^{\theta,p}(0,T,w_\kappa;\rX_{1-\theta}) \subset \rL^p(0,T,w_\kappa;\rX_1) \cap \rC([0,T];\rX_{1-\frac{1+\kappa}{p},p}).
\end{equation*}
If $p = 2$, we use $\rL^2(0,T;\rX_1) \cap \rC([0,T];\rX_{\nicefrac{1}{2}})$.

Below, we capture the main assumption on the nonlinearity $F$ in \eqref{eq:bilin SPDE}.
Let us observe that the assumption on the $\Hinfty$-calculus of the operator $A$ is not necessary at this stage, but we include it for convenience, since by \autoref{lem:rel of Hinfty with other concepts}(b) it allows us to characterize the fractional power domains by the complex interpolation spaces.
For completeness, we also briefly recall the assumptions on $\rX_0$ and $\rX_1$.

\begin{asu}\label{ass:main ass bilin setting}
Assume that $\rX_0$ is a UMD space with type $2$, and that there is $\lambda \ge 0$ such that $\lambda + A \in \Hinfty(\rX_0)$ and $\rD(A) = \rX_1$.
Moreover, let $p \in [2,\infty)$, $\kappa \in [0,\nicefrac{p}{2}-1) \cup \{0\}$ and $\beta \in (1-\frac{1+\kappa}{p},1)$ with $\frac{1+\kappa}{p} \le 2(1-\beta)$, and suppose that $F(u) = G(u,u)$, where $G \colon \rX_\beta \times \rX_\beta \to \rX_0$ is bilinear and bounded.
Finally, assume that $g$ is as in \autoref{def:stoch max reg}.
\end{asu}

We say that a couple $(p,\kappa)$ is {\em critical} provided $\frac{1+\kappa}{p} = 2(1-\beta)$.
Accordingly, $(p,\kappa)$ is called {\em subcritical} if strict inequality holds in \autoref{ass:main ass bilin setting}, i.e., $\frac{1+\kappa}{p} < 2(1-\beta)$.

Before presenting the main abstract results, we make precise the notion of an $\rL_\kappa^p$-strong solution of the bilinear SPDE \eqref{eq:bilin SPDE}.

\begin{defn}\label{def:Lpkappa strong sol}
Let \autoref{ass:main ass bilin setting} be fulfilled for some $p \in [2,\infty)$ and $\kappa \in [0,\nicefrac{p}{2}-1) \cup \{0\}$.
Then we say that a pair $(u,\sigma)$ is an $\rL_\kappa^p$-strong solution of \eqref{eq:bilin SPDE} if $\sigma \colon \Omega \to [0,\infty)$ is a stopping time, $u \colon [0,\sigma] \to \rX_0$ is a strongly progressively measurable process with $u \in \rL^p(0,\sigma,w_\kappa;\rX_1) \cap \rC([0,\sigma];\rX_{1-\frac{1+\kappa}{p},p})$, and {a.s.} for all $t \in [0,\sigma]$, we have
\begin{equation*}
    u(t) - u_0 + \int_0^t A u(s) \rd s = \int_0^t F(u(s)) \rd s + \int_0^t \bfone_{[0,\sigma]}(s)[B(u(s)) + g(s)] \rd W_\rH(s).
\end{equation*}
\end{defn}

In the subsequent definition, we specify the concepts of an $\rL_\kappa^p$-local solution, uniqueness, an $\rL_\kappa^p$-maximal solution and the global character of a solution.

\begin{defn}
Let \autoref{ass:main ass bilin setting} be fulfilled for some $p \in [2,\infty)$ and $\kappa \in [0,\nicefrac{p}{2}-1) \cup \{0\}$.
\begin{enumerate}[(a)]
    \item We refer to a pair $(u,\sigma)$ as an $\rL_\kappa^p$-local solution to \eqref{eq:bilin SPDE} provided $\sigma \colon \Omega \to [0,\infty)$ is a stopping time, $u \colon [0,\sigma) \to \rX_0$ is a strongly progressively measurable process, and there is an increasing sequence of stopping times $(\sigma_n)_{n \ge 1}$ such that $\lim_{n \to \infty} \sigma_n = \sigma$ a.s., and $(\left. u\right|_{[0,\sigma_n]},\sigma)$ is an $\rL_\kappa^p$-strong solution to \eqref{eq:bilin SPDE}.
    The sequence $(\sigma_n)_{n \ge 1}$ is then also called a localizing sequence for $(u,\sigma)$.
    \item An $\rL_\kappa^p$-local solution $(u,\sigma)$ to \eqref{eq:bilin SPDE} is said to be unique if for each $\rL_\kappa^p$-local solution $(v,\tau)$, one has a.s.\ $u = v$ on $[0,\sigma \land \tau)$.
    \item We say that an $\rL_\kappa^p$-local solution $(u,\sigma)$ to \eqref{eq:bilin SPDE} is $\rL_\kappa^p$-maximal if for any other unique $\rL_\kappa^p$-local solution $(v,\tau)$ to \eqref{eq:bilin SPDE}, it holds that a.s.\ $\tau \le \sigma$ and $u = v$ on $[0,\tau)$.
    \item An $\rL_\kappa^p$-local solution $(u,\sigma)$ to \eqref{eq:bilin SPDE} is referred to as global provided $\sigma = \infty$ a.s.\ holds true.
\end{enumerate}
\end{defn}

We are now in the position to state the main result of this subsection on the local strong well-posedness of~\eqref{eq:bilin SPDE}.
The result includes the local existence and uniqueness of a strong solution as well as the local continuity.
Note that the result below is implied by the results in \cite[Section~4]{AV:22a}, see also \cite[Theorem~4.7 and~4.8]{AV:25}, upon making the following observations.
For $F(u) = G(u,u) \colon \rX_\beta \times \rX_\beta \to \rX_0$ bilinear and bounded, we get the estimate
\begin{equation*}
    \| F(u) - F(v) \|_{\rX_0} \le C(\| u \|_{\rX_\beta} + \| v \|_{\rX_\beta}) \| u - v \|_{\rX_\beta},
\end{equation*}
so we are in the framework of \cite[Assumption~4.1]{AV:25} with $\rho = 1$.
We remark that we focus on the situation of complex interpolation spaces here.
It can then be verified that the conditions of the parameters $p$, $\kappa$ and $\beta$ are as specified in \autoref{ass:main ass bilin setting}.

\begin{prop}\label{prop:loc ex result bilin SPDE}
Consider $p \in [2,\infty)$ and $\kappa \in [0,\nicefrac{p}{2}-1) \cup \{0\}$, and suppose that \autoref{ass:main ass bilin setting} is valid.
Further, assume that $A \in \cL(\rX_1,\rX_0)$ and $B \in \cL(\rX_1,\gamma(\rH,\rX_{\nicefrac{1}{2}}))$ fulfill $(A,B) \in \cSMR_{p,\kappa}^\bullet$.
Then for all $u_0 \in \rL_{\frF_0}^0(\Omega;\rX_{1-\frac{1+\kappa}{p},p})$, there is an $\rL_\kappa^p$-maximal solution $(u,\sigma)$ to \eqref{eq:bilin SPDE} such that $\sigma > 0$ a.s.\ holds true.
Besides, we get the subsequent additional properties:
\begin{enumerate}[(a)]
    \item For every localizing sequence $(\sigma_n)_{n \ge 1}$ for $(u,\sigma)$, we have
    \begin{itemize}
        \item if $p > 2$ and $\kappa \in [0,\nicefrac{p}{2}-1)$, then for every $n \ge 1$ and $\theta \in [0,\nicefrac{1}{2})$, we get
        \begin{equation*}
            u \in \rH^{\theta,p}(0,\sigma_n,w_\kappa;\rX_{1-\theta}) \cap \rC([0,\sigma_n];\rX_{1-\frac{1+\kappa}{p},p}) \enspace \text{a.s.},
        \end{equation*}
        and we obtain the instantaneous regularization reading as
        \begin{equation*}
            u \in \rH^{\theta,p}((0,\sigma);\rX_{1-\theta}) \cap \rC((0,\sigma);\rX_{1-\frac{1}{p},p}) \enspace \text{a.s.};
        \end{equation*}
        \item if $p = 2$, then for each $n \ge 1$, we have
        \begin{equation*}
            u \in \rL^2(0,\sigma_n;\rX_1) \cap \rC([0,\sigma_n];\rX_{\frac{1}{2}}) \enspace \text{a.s.}.
        \end{equation*}
    \end{itemize}
    \item Consider a strongly $\frF_0$-measurable initial datum $v_0 \colon \Omega \to \rX_{1-\frac{1+\kappa}{p},p}$.
    If $(v,\tau)$ is the $\rL_\kappa^p$-maximal solution to \eqref{eq:bilin SPDE} with initial value $v_0$, then a.s.\ on the set $\{u_0 = v_0\}$, we have $\tau = \sigma$ and $u = v$ on $[0,\sigma \land \tau)$.
    \item Assume that $u_0 \in \rL_{\frF_0}^p(\Omega;\rX_{1-\frac{1+\kappa}{p},p})$, and consider the resulting maximal $\rL_\kappa^p$-solution $(u,\sigma)$ to~\eqref{eq:bilin SPDE}.
    Then there are constants $C_0$, $T_0$, $\eps_0 > 0$ and stopping times $\sigma_0$, $\sigma_1 \in (0,\sigma]$ a.s.\ such that the following holds true.

    For every $v_0 \in \rL_{\frF_0}^p(\Omega;\rX_{1-\frac{1+\kappa}{p},p})$ with $\E \| u_0 - v_0 \|_{\rX_{1-\frac{1+\kappa}{p},p}}^p \le \eps_0$, the maximal $\rL_\kappa^p$-solution~$(v,\tau)$ to \eqref{eq:bilin SPDE} with initial data $v_0$ satisfies that there is a stopping time $\tau_0 \in (0,\tau]$ a.s.\ such that for all $t \in [0,T_0]$ and $\delta > 0$, we have
    \begin{equation*}
        \begin{aligned}
            \bP\Bigl(\sup_{r \in [0,t]} \| u(r) - v(r) \|_{\rX_{1-\frac{1+\kappa}{p},p}} \ge \delta,\, \sigma_0 \wedge \tau_0 > t\Bigr)
            &\le \frac{C_0}{\delta^p} \cdot \E \| u_0 - v_0 \|_{\rX_{1-\frac{1+\kappa}{p},p}}^p,\\
            \bP\Bigl(\| u - v \|_{\rL^p(0,t,w_\kappa;\rX_1)} \ge \delta, \, \sigma_0 \wedge \tau_0 > t\Bigr)
            &\le \frac{C_0}{\delta^p} \cdot \E \| u_0 - v_0 \|_{\rX_{1-\frac{1+\kappa}{p},p}}^p, \tand\\
            \bP(\sigma_0 \wedge \tau_0 \le t)
            &\le C_0 \bigl[\E \| u_0 - v_0 \|_{\rX_{1-\frac{1+\kappa}{p},p}}^p + \bP(\sigma_1 \le t) \bigr].
        \end{aligned}
    \end{equation*}
\end{enumerate}
\end{prop}

A few words on the meaning of the assertion of~(c) are in order.
We observe that the stopping time $\tau_0$ depends on $(u_0,v_0)$.
The first two estimates somewhat show the continuous dependence of the solution~$(u,\sigma)$ on the initial data $u_0$.
On the other hand, the last estimate in~(c) gives a measure of the time interval on which the continuity estimates are valid.
Note that the right-hand side only depends on~$v_0$ but not on $v$, and $\{\tau_0 \le t\}$ has a small probability if $t$ is close to zero and $v_0$ is close to $u_0$.

We conclude this section by providing blow-up criteria for $\rL_\kappa^p$-maximal solutions $(u,\sigma)$ to \eqref{eq:bilin SPDE}.
A blow-up means that $\bP(\sigma < \infty) > 0$, and a blow-up criterion can be rephrased as $\bP(\sigma < \infty, \, u \in \cC_0) = 0$, where $\cC_0$ encodes certain properties of $u$.
Typically, this corresponds to $u$ belonging to a certain function space or the existence of specific limits as time approaches the stopping time $\sigma$.
The following result can be found in \cite[Section~4]{AV:22b}, see also \cite[Theorem~5.1 and 5.2]{AV:25}.
We distinguish between the critical case and the subcritical case

\begin{prop}\label{prop:blow-up crit bilin SPDE}
Assume that the conditions from \autoref{prop:loc ex result bilin SPDE} are satisfied for $p \in [2,\infty)$ as well as~$\kappa \in [0,\nicefrac{p}{2}-1) \cup \{0\}$, and consider the $\rL_\kappa^p$-maximal solution $(u,\sigma)$ to \eqref{eq:bilin SPDE}.
\begin{enumerate}[(a)]
    \item In the critical regime $\frac{1+\kappa}{p} = 2(1-\beta)$, it holds that $\bP\bigl(\sigma < \infty, \, \lim_{t \nearrow \sigma} u(t) \text{ exists in } \rX_{1-\frac{1+\kappa}{p},p}\bigr) = 0$ and $\bP\bigl(\sigma < \infty, \, \sup_{t \in [0,\sigma)} \| u(t) \|_{\rX_{1-\frac{1+\kappa}{p},p}} + \| u \|_{\rL^p(0,\sigma;\rX_{1-\frac{\kappa}{p}})}\bigr) = 0$.
    \item In the subcritical case $\frac{1+\kappa}{p} < 2(1-\beta)$, we have $\bP\bigl(\sigma < \infty, \, \sup_{t \in [0,\sigma)} \| u(t) \|_{\rX_{1-\frac{1+\kappa}{p},p}} < \infty\bigr) = 0$.
\end{enumerate}
\end{prop}

Let us observe that the norms in the second estimate in~(a) are well-defined thanks to the weighted Sobolev embedding $\rH_\loc^{\frac{\kappa}{p},p}([0,\sigma),w_\kappa;\rX_{1-\frac{\kappa}{p}}) \hookrightarrow \rL_\loc^p([0,\sigma);\rX_{1-\frac{\kappa}{p}})$ a.s., see \cite[Proposition~2.7]{AV:22a}.

\section{Main results}\label{sec:main results}

In this section, we present the main results on the analysis of the stochastic fluid-rigid body interaction problem.
In fact, we state the local strong well-posedness of \eqref{eq:stoch forced fsi problem intro}, including the continuous dependence on the initial data, and we also give blow-up criteria.

Before asserting the main results, we require some preparation in terms of notation.
First, we define function spaces that turn out to be the real and complex interpolation spaces in the present setting.
For this purpose, let us also remind the space
$\rL_\sigma^q(\cF_0) \coloneqq \overline{\{f \in \rC_\mathrm{c}^\infty(\cF_0)^3 : \mdiv f = 0\}}^{\| \cdot \|_{\rL^q}}$ of solenoidal vector fields as well as the associated Helmholtz projection $\cP$.
For more details, see \autoref{lem:props Helmholtz and Stokes}.

Next, we introduce the spaces $\rX_0$ and $\rX_1$ given by
\begin{equation*}
    \begin{aligned}
        \rX_0 
        &\coloneqq \{(v,\ell,\omega) \in \rL^q(\cF_0)^3 \times \R^3 \times \R^3 : \mdiv v = 0 \tin \cF_0, \, v \cdot n = (\ell + \omega \times y) \cdot n \ton \del \cS_0\} \tand\\
        \rX_1 
        &\coloneqq \left\{(v,\ell,\omega) \in \rW^{2,q}(\cF_0)^3 \times \R^3 \times \R^3 : \mdiv v = 0 \tin \cF_0, \, v = \ell + \omega \times y \ton \del \cS_0\right\}.
    \end{aligned}
\end{equation*}
We provide the space 
\begin{equation}\label{eq:X_1/2}
    \rX_{\nicefrac{1}{2}} = [\rX_0,\rX_1]_{\nicefrac{1}{2}} = \bigl\{(v,\ell,\omega) \in \rH^{1,q}(\cF_0)^3 \times \R^3 \times \R^3 : \mdiv v = 0 \tin \cF_0, \, v = \ell + \omega \times y \ton \del \cS_0\bigr\}.
\end{equation}
\begin{defn}\label{def:function spaces}
Denote by $w$ the principle variable, i.e., $w = (v,\ell,\omega)$.
For $s \ge 0$ and $q \in (1,\infty)$, we set
\begin{equation*}
    \rH_{c,\sigma}^{s,q}(\cF_0) \coloneqq
    \left\{
    \begin{aligned}
        &\{w \in \rH^{s,q}(\cF_0)^3 \times \R^3 \times \R^3 : \mdiv v = 0 \tin \cF_0, \, v = \ell + \omega \times y \ton \del \cS_0\}, &&s > \frac{1}{q},\\
        &[\rX_0,\rX_{\nicefrac{1}{2}}]_{\nicefrac{1}{q}}, &&s = \frac{1}{q},\\
        &\{w \in \rH^{s,q}(\cF_0)^3 \times \R^3 \times \R^3 : \mdiv v = 0 \tin \cF_0, \, v \cdot n = (\ell + \omega \times y) \cdot n \ton \del \cS_0\}, &&s < \frac{1}{q}.
    \end{aligned}
    \right.
\end{equation*}
Likewise, for $s > 0$ and $p$, $q \in (1,\infty)$, we define
\begin{equation*}
    \rB_{qp,c,\sigma}^{s}(\cF_0) \coloneqq
    \left\{
    \begin{aligned}
        &\{w \in \rB_{qp}^{s}(\cF_0)^3 \times \R^3 \times \R^3 : \mdiv v = 0 \tin \cF_0, \, v = \ell + \omega \times y \ton \del \cS_0\}, &&s > \frac{1}{q},\\
        &(\rX_0,\rX_{\nicefrac{1}{2}})_{\nicefrac{1}{q},p}, &&s = \frac{1}{q},\\
        &\{w \in \rB_{qp}^{s}(\cF_0)^3 \times \R^3 \times \R^3 : \mdiv v = 0 \tin \cF_0, \, v \cdot n = (\ell + \omega \times y) \cdot n \ton \del \cS_0\}, &&s < \frac{1}{q}.
    \end{aligned}
    \right.
\end{equation*}
\end{defn}
The lemma below asserts that the complex and real interpolation spaces of $\rX_0$ and $\rX_1$ indeed coincide with the spaces from \autoref{def:function spaces}.
The proof can be found in \autoref{sec:proof main results}.

\begin{lem}\label{lem:interpol spaces}
Let $\theta \in (0,1)$ and $p$, $q \in (1,\infty)$.
Then $[\rX_0,\rX_1]_\theta = \rH_{c,\sigma}^{2 \theta,q}(\cF_0)$ and $(\rX_0,\rX_1)_{\theta,p} = \rB_{qp,c,\sigma}^{2 \theta}(\cF_0)$, where $\rH_{c,\sigma}^{2 \theta,q}(\cF_0)$ and $\rB_{qp,c,\sigma}^{2 \theta}(\cF_0)$ are as made precise in \autoref{def:function spaces}.
\end{lem}

We are now in the position to make precise the noise terms in \eqref{eq:stoch forced fsi problem intro}.
Indeed, let $(W_t^n)_{n \ge 1}$ be a family of standard independent Brownian motions, and consider $b = (b_n)_{n \ge 1} \colon \cF_0 \to \ell^2$ measurable and bounded as well as~$f = (f_n)_{n \ge 1} = ((f_v^n,f_\ell^n,f_\omega^n)^\top)_{n \ge 1} \in \rL_{\frP}^p(\Omega;\rL^p(I_T,w_\kappa;\gamma(\ell^2,\rX_{\nicefrac{1}{2}})))$, where we recall $\rX_{\nicefrac{1}{2}}$ from~\eqref{eq:X_1/2} and the definition of the space $\rL_{\frP}^p(\Omega;\rL^p(I_T,w_\kappa;\gamma(\rH,\rX_{\nicefrac{1}{2}})))$ from \autoref{ssec:stoch max reg and semilin SPDEs}.

Below, we assert the first main result on the local strong well-posedness of \eqref{eq:stoch forced fsi problem intro}.
For convenience, we state the cases $p > 2$ and $p = 2$ separately, and we start with the former case.

\begin{thm}\label{thm:loc strong wp stoch fsi}
Let $p \in (2,\infty)$, $q \in [2,\infty)$ and $\kappa \in [0,\frac{p}{2}-1)$, and recall $\rX_{1-\frac{1+\kappa}{p},p}$ from \autoref{lem:interpol spaces}.
Then there is $\eps > 0$ so that if $\| (b_n)_{n \ge 1} \|_{\rW^{1,\infty}(\cF_0;\ell^2)} < \eps$, then for all $w_0 = (\cP v_0,\ell_0,\omega_0) \in \rL_{\frF_0}^0(\Omega;\rX_{1-\frac{1+\kappa}{p},p})$, there exists an $\rL_\kappa^p$-maximal solution $(w,\sigma)$ to \eqref{eq:stoch forced fsi problem intro}, with $\sigma > 0$ a.s..
Moreover,
\begin{enumerate}[(a)]
    \item for each localizing sequence $(\sigma_n)_{n \ge 1}$, and for all $n \ge 1$ and $\theta \in [0,\frac{1}{2})$, we have
    \begin{equation*}
        w \in \rH^{\theta,p}(0,\sigma_n,w_\kappa;\rX_{1-\theta}) \cap \rC([0,\sigma_n];\rX_{1-\frac{1+\kappa}{p},p}) \enspace \text{a.s.},
    \end{equation*}
    and the solution regularizes instantaneously, i.e.,
    \begin{equation*}
        w \in \rH^{\theta,p}((0,\sigma);\rX_{1-\theta}) \cap \rC((0,\sigma);\rX_{1-\frac{1}{p},p}) \enspace \text{a.s.};
    \end{equation*}
    \item for a strongly $\frF_0$-measurable initial datum $\Tilde{w}_0 \colon \Omega \to \rX_{1-\frac{1+\kappa}{p},p}$, denoting by $(\Tilde{w},\tau)$ the resulting $\rL_\kappa^p$-maximal solution to \eqref{eq:stoch forced fsi problem intro}, a.s.\ on $\{w_0 = \Tilde{w}_0\}$, we obtain $\tau = \sigma$ and $w = \Tilde{w}$ on $[0,\sigma \land \tau)$;
    \item let $w_0 \in \rL_{\frF_0}^p(\Omega;\rX_{1-\frac{1+\kappa}{p},p})$, and denote by $(w,\sigma)$ the emerging $\rL_\kappa^p$-maximal solution to \eqref{eq:stoch forced fsi problem intro}.
    Then there exist $C_0$, $T_0$, $\eps_0 > 0$ and stopping times $\sigma_0$, $\sigma_1 \in (0,\sigma]$ a.s.\ so that the following is valid.
    For each $\Tilde{w}_0 \in \rL_{\frF_0}^p(\Omega;\rX_{1-\frac{1+\kappa}{p},p})$ with $\E \| w_0 - \Tilde{w}_0 \|_{\rX_{1-\frac{1+\kappa}{p},p}}^p \le \eps_0$, the $\rL_\kappa^p$-maximal solution $(\Tilde{w},\tau)$ to \eqref{eq:stoch forced fsi problem intro} with initial data $\Tilde{w}_0$ satisfies that there exists a stopping time $\tau_0 \in (0,\tau]$ a.s.\ so that for all $t \in [0,T_0]$ and $\delta > 0$, it holds that
    \begin{equation*}
        \begin{aligned}
            \bP\Bigl(\sup_{r \in [0,t]} \| w(r) - \Tilde{w}(r) \|_{\rX_{1-\frac{1+\kappa}{p},p}} \ge \delta,\, \sigma_0 \wedge \tau_0 > t\Bigr)
            &\le \frac{C_0}{\delta^p} \cdot \E \| w_0 - \Tilde{w}_0 \|_{\rX_{1-\frac{1+\kappa}{p},p}}^p,\\
            \bP\Bigl(\| w - \Tilde{w} \|_{\rL^p(0,t,w_\kappa;\rX_1)} \ge \delta, \, \sigma_0 \wedge \tau_0 > t\Bigr)
            &\le \frac{C_0}{\delta^p} \cdot \E \| w_0 - \Tilde{w}_0 \|_{\rX_{1-\frac{1+\kappa}{p},p}}^p, \tand\\
            \bP(\sigma_0 \wedge \tau_0 \le t)
            &\le C_0 \bigl[\E \| w_0 - \Tilde{w}_0 \|_{\rX_{1-\frac{1+\kappa}{p},p}}^p + \bP(\sigma_1 \le t) \bigr].
        \end{aligned}
    \end{equation*}
\end{enumerate}
\end{thm}

Let us briefly comment on the assertions of \autoref{thm:loc strong wp stoch fsi}.
In~(a), we discuss the spatial and temporal regularity of the solution, part~(b) provides a localization result, while~(c) can be somewhat regarded as a result on the continuous dependence on the initial data.

Next, we elaborate on the situation in the Hilbert space case.

\begin{thm}\label{thm:loc strong wp stoch fsi Hilbert space}
If $\| (b_n)_{n \ge 1} \|_{\rW^{1,\infty}(\cF_0;\ell^2)} < \eps$ for sufficiently small $\eps > 0$, then for all strongly $\frF_0$-measurable $w_0 = (\cP v_0,\ell_0,\omega_0)$ with
\begin{equation*}
    w_0 \in \rX_{\nicefrac{1}{2}} = \{(\cP v_0,\ell_0,\omega_0) \in \rH^1(\cF_0)^3 \cap \rL_\sigma^2(\cF_0) \times \R^3 \times \R^3 : v_0 = \ell_0 + \omega_0 \times y \ton \del \cS_0\},
\end{equation*}
there is an $\rL^2$-maximal solution $(w,\sigma)$ to \eqref{eq:stoch forced fsi problem intro} such that $\sigma > 0$ a.s.\ holds true.
Furthermore, recalling
\begin{equation*}
    \rX_1 = \{(\cP v,\ell,\omega) \in \rH^2(\cF_0)^3 \cap \rL_\sigma^2(\cF_0) \times \R^3 \times \R^3 : v = \ell + \omega \times y \ton \del \cS_0\},
\end{equation*}
for every localizing sequence $(\sigma_n)_{n \ge 1}$ and $n \ge 1$, we get $w \in \rL^2(0,\sigma_n;\rX_1) \cap \rC([0,\sigma_n];\rX_{\frac{1}{2}})$ a.s..
In addition, replacing $\rX_{1-\frac{1+\kappa}{p},p}$ by the above $\rX_{\nicefrac{1}{2}}$, similar assertions as in \autoref{thm:loc strong wp stoch fsi}(b) and~(c) are valid.
\end{thm}

The remark below concerns the regularity of the initial data and especially discusses the criticality for a certain range of $q$.
We also comment on the smallness assumption on the transport noise.

\begin{rem}\label{rem:crit of init data}
\begin{enumerate}[(a)]
    \item For $q \in [2,3)$, we can consider the {\em critical} weight $\kappa_c = \frac{3p}{2}(1-\frac{1}{q})-1$, and the space for the initial data in \autoref{thm:loc strong wp stoch fsi} becomes $\rX_{\frac{3}{2q}-\frac{1}{2},p}$.
    For $q > 2$, we thus take into account $\frF_0$-measurable
    \begin{equation}\label{eq:init data crit weights}
        w_0 = (\cP v_0,\ell_0,\omega_0) \in \rB_{qp}^{\frac{3}{q}-1}(\cF_0)^3 \cap \rL_\sigma^q(\cF_0) \times \R^3 \times \R^3 : v_0 \cdot n = (\ell_0 + \omega_0 \times y) \cdot n \ton \del \cS_0,
    \end{equation}
    or $w_0 \in (\rX_0,\rX_{\nicefrac{1}{2}})_{\nicefrac{1}{2},p}$ in the case $q = 2$.
    For the fluid component, this aligns with the case of the Navier-Stokes equations in the strong, {\em deterministic} setting as studied in \cite[Section~3]{PW:17}.
    In \cite{AV:24b}, the range of admissible $q$'s for the stochastic Navier-Stokes equations on the torus was extended by considering a weaker setting to lower the regularity of the initial data.
    This would also be tempting here, but due to the coupled nature of the present problem, it is a priori not clear how to interpret a weak fluid-structure operator, and how to consider the interaction problem appropriately in the weak setting.
    Further investigation in this direction is left to future study.
    \item For deterministic fluid-rigid ball interaction problems, the regularity of the initial data stated in \autoref{thm:loc strong wp stoch fsi Hilbert space} matches the regularity reported in \cite[Theorem~2.2]{Takahashi:03} or \cite[Theorem~1.4]{MT:18} in the Hilbert space case.
    However, the theory of critical weights from \cite{PW:17, PSW:18}, which is applicable thanks to the bounded $\Hinfty$-calculus of the fluid-structure operator as demonstrated below in  \autoref{thm:bdd H00-calculus of fluid-structure op incompr Newtonian}, allows to consider initial data as in \eqref{eq:init data crit weights} and thus enlarges the space of admissible initial data compared to \cite{Takahashi:03, MT:18}.
    We omit the details here and only note that the result can be derived in a similar way as in \cite[Section~3]{PW:17}.
    \item It seems that the smallness assumption on the transport noise in \autoref{thm:loc strong wp stoch fsi} and \autoref{thm:loc strong wp stoch fsi Hilbert space} is indispensable at the moment.
    The consideration of non-small transport noise for 3D Navier-Stokes equations constitutes an open problem even in the setting without rigid body, see \cite[Problem~9]{AV:25}.
\end{enumerate}
\end{rem}

After addressing the well-posedness results, we now deal with the blow-up criteria.
Again, we distinguish between the case $p > 2$ and the Hilbert space case.

\begin{thm}\label{thm:blow-up crit stoch fsi}
Suppose that the assumptions from \autoref{thm:loc strong wp stoch fsi} are satisfied for $p > 2$, $q \in [2,\infty)$, and $\kappa \in [0,\nicefrac{p}{2}-1)$, and consider the maximal $\rL_\kappa^p$-solution $(w,\sigma)$ to \eqref{eq:stoch forced fsi problem intro}.
\begin{enumerate}[(a)]
    \item In the situation $q \in [2,3)$, we have $\bP\bigl(\sigma < \infty, \, \lim_{t \nearrow \sigma} w(t) \text{ exists in } \rX_{\frac{3}{2q}-\frac{1}{2},p}\bigr) = 0$ as well as $\bP\bigl(\sigma < \infty, \, \sup_{t \in [0,\sigma)} \| w(t) \|_{\rX_{\frac{3}{2q}-\frac{1}{2},p}} + \| w \|_{\rL^p(0,\sigma;\rX_{\frac{1}{p}+\frac{3}{2q}-\frac{1}{2}})}\bigr) = 0$.
    \item If $q \ge 3$, it holds that $\bP\bigl(\sigma < \infty, \, \sup_{t \in [0,\sigma)} \| w(t) \|_{\rX_{1-\frac{1+\kappa}{p},p}} < \infty\bigr) = 0$.
\end{enumerate}
\end{thm}

We conclude this section with the blow-up criterion in the Hilbert space case.
Let us observe that the case $p = q = 2$ is subcritical in 3D.

\begin{thm}\label{thm:blow-up crit stoch fsi Hilbert space}
Under the assumptions of \autoref{thm:loc strong wp stoch fsi Hilbert space}, for the $\rL^2$-maximal solution $(w,\sigma)$ resulting from the latter theorem, it is valid that $\bP\bigl(\sigma < \infty, \, \sup_{t \in [0,\sigma)} \| w(t) \|_{\rX_{\nicefrac{1}{2}}} < \infty\bigr) = 0$.
\end{thm}

\section{Bounded $\Hinfty$-calculus of the fluid-structure operator}\label{sec:bdd Hoo-calculus fluid-structure operator}

In this section, we introduce the fluid-structure operator associated to the fluid-rigid body interaction problem of a viscous, incompressible, Newtonian fluid, and we establish its bounded $\Hinfty$-calculus.
Moreover, we derive further functional analytic properties such as the stochastic maximal regularity from there.
With regard to the setup of the stochastic fluid-rigid body interaction problem~\eqref{eq:stoch forced fsi problem intro}, we focus on the case of the fluid and the rigid body filling the entire $\R^3$.
We will comment on the associated bounded domain case in \autoref{sec:concl rems & further discussion}.
In contrast to the situation in~\eqref{eq:stoch forced fsi problem intro}, the analysis in this section is carried out for a rigid body {\em of arbitrary shape} with boundary of class $\rC^3$.

As a preparation, we collect some properties of the Helmholtz decomposition and the Stokes operator domains with compact boundary.
The first three assertions can be found in \cite[Theorem~5.3 and Theorem~1.4]{SS:92}, while we refer to \cite[Theorem~3]{NS:03} for the last assertion.
The first three assertions also hold if the boundary of the domain is only in $\rC^1$.
For simplicity, and as this is the relevant case for our application, we focus on the 3D case, but the results remain valid for general $n \ge 2$.

\begin{lem}\label{lem:props Helmholtz and Stokes}
Let $q \in (1,\infty)$, consider a domain $\cF_0 \subset \R^3$ with compact boundary of class $\rC^3$, recall~$\rL_\sigma^q(\cF_0)$ from \autoref{sec:main results}, and define the spaces $\rW_{\mdiv}^q(\cF_0) \coloneqq \{f \in \rL^q(\cF_0)^3 : \mdiv f \in \rL^q(\cF_0)\}$ as well as~$\rG^q(\cF_0) \coloneqq \{\nabla f : f \in \rL_\mathrm{loc}^q(\overline{\cF_0}), \enspace \nabla f \in \rL^q(\cF_0)^3\}$.
\begin{enumerate}[(a)]
    \item The normal trace $f \mapsto \gamma_n f \coloneqq \varphi|_{\del \cF_0} \cdot n$, for the outer unit normal vector $n$, extends to a continuous linear operator from $\rW_{\mdiv}^q(\cF_0)$ to $\rW^{-\nicefrac{1}{q},q}(\cF_0)$.
    \item The space $\rL^q(\cF_0)^3$ can be decomposed into $\rL^q(\cF_0)^3 = \rL_\sigma^q(\cF_0) \oplus \rG^q(\cF_0)$.
    In other words, for every $f \in \rL^q(\cF_0)^3$, there exist unique $u \in \rL_\sigma^q(\cF_0)$ and $\nabla p \in \rG^q(\cF_0)$ such that $f = u + \nabla p$.
    This decomposition is known as Helmholtz decomposition.
    The associated projection onto $\rL_\sigma^q(\cF_0)$ is denoted by $\cP$, and $p$ solves the weak Neumann problem
    \begin{equation*}
        \Delta p = \mdiv u, \tin \cF_0, \enspace \del_n p = u \cdot n, \ton \del \cF_0.
    \end{equation*}
    \item It holds that $\rL_\sigma^q(\cF_0) = \{u \in \rW_{\mdiv}^q(\cF_0) : \mdiv u = 0 \tin \cF_0 \tand u \cdot n = 0, \ton \del \cF_0\}$.
    \item Define the Stokes operator $A_0$ on $\rL_\sigma^q(\cF_0)$ with homogeneous Dirichlet boundary conditions by
    \begin{equation}\label{eq:Stokes op}
        A_0 u \coloneqq \cP u, \tfor u \in \rD(A_0) \coloneqq \rW^{2,q}(\cF_0)^3 \cap \rW_0^{1,q}(\cF_0)^3 \cap \rL_\sigma^q(\cF_0).
    \end{equation}
    Then $-A_0 \in \Hinfty(\rL_\sigma^q(\cF_0))$ with $\Hinfty$-angle $\phi_{A_0}^\infty = 0$.
\end{enumerate}
\end{lem}

Let us first describe the setting.
As mentioned above, we assume here that the fluid-rigid body system fills the whole space~$\R^3$, and we consider the fixed domain case, i.e., we suppose that the rigid body is at rest.
By $\cS_0$, we denote the domain occupied by the rigid body, and we assume its boundary $\del \cS_0$ to be of class $\rC^3$.
The fluid then occupies $\cF_0 \coloneqq \R^3 \setminus \cS_0$.
For $T \in (0,\infty]$, the model variables are the fluid velocity $u \colon (0,T) \times \cF_0 \to \R^3$, the pressure $p \colon (0,T) \times \cF_0 \to \R$, the translational velocity of the rigid body~$\ell \colon (0,T) \to \R^3$, and the rotational velocity $\omega \colon (0,T) \to \R^3$.
Moreover, recall that $\mS$ and $J_0$ represent the body mass and inertia tensor, and recall the Cauchy stress tensor~$\rT(u,p)$ from \eqref{eq:Cauchy stress tensor}.

For a shift $\lambda \in \C$, for $n$ denoting the unit outward normal to $\del \cF_0$, and for suitable right-hand sides $f_1$, $f_2$ and $f_3$ and initial data $u_0$, $\ell_0$ and $\omega_0$, the linearized fluid-rigid body interaction problem reads as
\begin{equation}\label{eq:lin ext}
\left\{
    \begin{aligned}
        \partial_t u + \lambda u  -\Delta u + \nabla {p} &=f_1, \enspace \mdiv {u} = 0, &&\tin (0,T) \times \cF_0,\\
        \mS({\ell})' + \lambda \ell + \int_{\del \cS_0} \rT(u,p) n \rd \Gamma &= f_2, &&\tin (0,T),\\ 
        J_0({\omega})' + \lambda \omega + \int_{\del \cS_0} y \times \rT(u,p) n \rd \Gamma &= f_3, &&\tin (0,T),\\
        u &={\ell}+{\omega}\times y, && \ton (0,T) \times \partial\cS_0,\\
        u(0) = u_0, \enspace \ell(0) &= \ell_0 \tand \omega(0) = \omega_0. 
    \end{aligned}
\right.
\end{equation}

For the following lifting procedure, we take into account a stationary problem capturing the interface condition.
More precisely, for $\lambda \in \C$ and $\ell$, $\omega \in \C^3$, we consider 
\begin{equation}\label{eq:stat problem ext}
    \left\{
    \begin{aligned}
        \lambda w -\Delta w + \nabla \psi
        &= 0, \enspace \mdiv w = 0, &&\tfor y \in \cF_0,\\
        w
        &= \ell + \omega \times y, &&\tfor y \in \del\cS_0.
    \end{aligned}
    \right.
\end{equation}
The lemma below, for which we refer to \cite[Proposition~4.1]{EMT:23}, addresses the solvability of \eqref{eq:stat problem ext}.

\begin{lem}\label{lem:sol to stat problem ext}
Let $\theta \in (0,\pi)$ and $q \in (1,\infty)$.
Then for all $\lambda \in \Sigma_\theta$ and $\ell$, $\omega \in \C^3$, there exists a unique solution $(w,\psi) \in \rW^{2,q}(\cF_0)^3 \times \widehat{\rW}^{1,q}(\cF_0)$ to \eqref{eq:stat problem ext}.
\end{lem}

Thanks to \autoref{lem:sol to stat problem ext}, for $\lambda \in \Sigma_\theta$, we define the family of operators
\begin{equation}\label{eq:lifting op D_lambda}
    D_\lambda \binom{\ell}{\omega} = \binom{D_{\lambda,u}(\ell,\omega)}{D_{\lambda,p}(\ell,\omega)} = \binom{w}{\psi},
\end{equation}
where $(w,\psi) \in \rW^{2,q}(\cF_0)^3 \times \widehat{\rW}^{1,q}(\cF_0)$ is the solution to \eqref{eq:stat problem ext}.
By construction, we have 
\begin{equation}\label{eq:mapping props D_lambda}
    D_\lambda \in \cL\bigl(\C^3 \times \C^3,\rW^{2,q}(\cF_0)^3 \times \widehat{\rW}^{1,q}(\cF_0)\bigr).
\end{equation}
In the sequel, we fix $\lambda_0 > 0$ and denote the associated lifting operator by $D_{\lambda_0} = \binom{D_{\lambda_0,u}}{D_{\lambda_0,p}}$.

The next goal is to reformulate the resolvent problem associated with the fluid part with inhomogeneous boundary conditions in terms of $\cP u$ and $(\Id - \cP)u$.
To this end, in addition to the lifting operator $D_\lambda$ from~\eqref{eq:lifting op D_lambda}, we first invoke the Neumann operator $\varphi \coloneqq N h$, assigning to each $h$ the solution $\varphi$ to
\begin{equation}\label{eq:Neumann problem ext}
    \Delta \varphi = \mdiv h, \tin \cF_0, \enspace \del_n \varphi = h \cdot n, \ton \del \cS_0.
\end{equation}
The following result on the solvability of this Neumann problem on the exterior domain is a consequence of \cite[Theorem~5.4]{SS:92}.

\begin{lem}\label{lem:Neumann problem ext}
Let $q \in (1,\infty)$.
Then for $h \in \rW_{\mdiv}^q(\cF_0)$, there is a unique solution $\nabla \varphi \in \rG^q(\cF_0)$ to \eqref{eq:Neumann problem ext}.
\end{lem}

We can also rephrase \autoref{lem:Neumann problem ext} as follows:
The Neumann operator $N \colon \rW_{\mdiv}^q(\cF_0) \to \widehat{\rW}^{1,q}(\cF_0)$, with~$N h \coloneqq \varphi$ representing the solution resulting from \autoref{lem:Neumann problem ext}, is well-defined.
Note that the present regularity is not enough to make sense of boundary integrals.
For this purpose, we modify the solution to the Neumann problem \eqref{eq:Neumann problem ext} so that its trace is well-defined on $\del \cS_0 = \del \cF_0$.
In fact, for $B_R$ denoting the ball in $\R^3$ with center zero and radius $R > 0$, we set $\cF_R \coloneqq \cF_0 \cap B_R$, and we fix $R > 0$ sufficiently large such that $\del \cF_0 \subset \del \cF_R$.
Thanks to $N h \in \widehat{\rW}^{1,q}(\cF_0)$ and the boundedness of $\cF_R$, it is especially valid that $N h \in \widehat{\rW}^{1,q}(\cF_R)$.
Hence, the expression 
\begin{equation}\label{eq:modified Neumann op ext}
    N_R \colon \rW_{\mdiv}^q(\cF_0) \to \widehat{\rW}^{1,q}(\cF_0) \cap \rL_0^q(\cF_R), \enspace N_R h \coloneqq N h - \frac{1}{|\cF_R|} \int_{\cF_R} N h \rd x,
\end{equation}
for $\rL_0^q(\cF_R)$ denoting the functions in $\rL^q(\cF_R)$ with average zero, is well-defined.
The advantage of the adjusted lifting operator $N_R$ becomes apparent in the lemma below.

\begin{lem}\label{lem:reg of Neumann op}
Consider $q \in (1,\infty)$ and $h \in \rW_{\mdiv}^q(\cF_0)$.
Then it holds that $N_R h \in \rW^{1-\nicefrac{1}{q},q}(\del \cS_0)$.
\end{lem}

\begin{proof}
We have already argued that $N h \in \rW^{1,q}(\cF_R)$.
As $N_R h$ and $N h$ only differ by a constant, and $\cF_R$ is bounded, we conclude that $N_R h \in \rW^{1,q}(\cF_R)$.
By standard trace theory, and by virtue of $\del \cS_0 = \del \cF_0 \subset \del \cF_R$ by construction, we deduce that $N_{R} h \in \rW^{1-\nicefrac{1}{q},q}(\del \cF_R) \subset \rW^{1-\nicefrac{1}{q},q}(\del \cS_0)$.
\end{proof}

As we will employ a perturbation argument in order to establish the boundedness of the $\Hinfty$-calculus, we also discuss the solvability of the weak Neumann problem \eqref{eq:Neumann problem ext} for less regular data.
For this purpose, we follow \cite[Section~5.2]{LM:06}.
More precisely, for $s \in (0,1)$, $q \in (1,\infty)$ and $q' \in (1,\infty)$ such that $\nicefrac{1}{q} + \nicefrac{1}{q'} = 1$, we consider $h \in \rH^{s+\nicefrac{1}{q}-1,q}(\cF_0)^3$.
It is then possible to extend $\mdiv h$ to a distribution in $\Tilde{\rH}^{s-1-\nicefrac{1}{q},q}(\cF_0)$, where
\begin{equation*}
    \Tilde{\rH}^{t,p}(\cF_0) \coloneqq \bigl\{f \in \rH^{t,p}(\cF_0) : \mathrm{supp}\, f \subset \overline{\cF_0}\bigr\}, \tfor t \in \R \tand p \in (1,\infty).
\end{equation*}
Moreover, the normal component $h \cdot n$ can be identified with an element of $\rW^{s-1,q}(\del \cF_0)^3$ via
\begin{equation*}
    \langle h \cdot n,\phi \rangle = \langle h, \nabla \Tilde{\phi} \rangle + \langle \mdiv u, \Tilde{\phi}\rangle
\end{equation*}
for all $\phi \in \rW^{1-s,q'}(\del \cF_0)$ extending to $\Tilde{\phi} \in \rH^{1-s+\nicefrac{1}{q'},q'}(\cF_0)$ in the trace sense, see also \cite[Lemma~5.3]{LM:06}.
By \cite[Lemma~5.5]{LM:06}, we have
\begin{equation*}
    \| h \cdot n \|_{\rW^{s-1,q}(\del \cF_0)} \le \| h \|_{\rH^{s+\nicefrac{1}{q}-1,q}(\cF_0)} + \| \mdiv h \|_{\rH^{s-\nicefrac{1}{q'}-1,q}}.
\end{equation*}
Furthermore, \cite[Theorem~5.8]{LM:06} implies that for all $q \in (\nicefrac{3}{2},\infty)$, there is $t \in (1,2-\nicefrac{1}{q})$ such that the Neumann problem \eqref{eq:Neumann problem ext} has a unique solution $\varphi \in \rW^{2-t,q}(\cF_0)^3$, and it satisfies the estimate
\begin{equation*}
    \| \varphi \|_{\rW^{2-t,q}(\cF_0)^3} \le C \cdot \| h \|_{\rH^{1-t,q}(\cF_0)}.
\end{equation*}
We stress that the restriction on $q > \nicefrac{3}{2}$ arises from the desire to circumvent the spatial weights in \cite[Theorem~5.8]{LM:06}.
They are introduced in \cite[Section~3]{LM:06} to show the boundedness of Newtonian potentials on unbounded domains, and the choice $t > 1$ then induces constraints on $q$ even though we consider an exterior $\rC^3$-domain instead of a Lipschitz one as in \cite{LM:06}.
Indeed, for $\delta_1 \in (-\nicefrac{3}{q},-2+\nicefrac{3}{q'})$ and $\delta_2 \in (-\nicefrac{3}{q'},-2+\nicefrac{3}{q})$, in view of \cite[Theorem~5.8]{LM:06}, we have to find $t \in (1,2-\nicefrac{1}{q})$ with
\begin{equation}\label{eq:unweighted case}
    0 = -\frac{t}{2}(\delta_1 + \delta_2) + \delta_1 + 2 (1-t)
\end{equation}
in order to bypass the weights.
Note that the above ranges of $\delta_1$ and $\delta_2$ can also be expressed completely in terms of $q$, namely, we have to find $\delta_1 \in (-\nicefrac{3}{q},1-\nicefrac{3}{q})$ and $\delta_2 \in (-3+\nicefrac{3}{q},-2+\nicefrac{3}{q})$.
In addition, we write $t = 1 + \eps$, where $\eps \in (0,1-\nicefrac{1}{q})$.
Thus, \eqref{eq:unweighted case} is equivalent with
\begin{equation}\label{eq:delta_2 in terms of delta_1}
    \frac{1+\eps}{2} \cdot \delta_2 + 2 \eps = \frac{1-\eps}{2} \cdot \delta_1, \tso \delta_2 = \frac{1-\eps}{1+\eps} \cdot \delta_1 - \frac{2}{1+\eps} \cdot 2 \eps.
\end{equation}
Let now $q \in (\nicefrac{3}{2},\infty)$.
We shall see in the sequel that it is possible to find suitable $\eps \in (0,1-\nicefrac{1}{q})$ and $\delta_1$, $\delta_2$ in the aforementioned ranges such that \eqref{eq:delta_2 in terms of delta_1}, or, equivalently, \eqref{eq:unweighted case} for $t \in (1,2-\nicefrac{1}{q})$ is satisfied.
For $q > \nicefrac{3}{2}$, we consider $\delta_1 \in (-\nicefrac{3}{q},1-\nicefrac{3}{q})$ and $\delta_2$ as given in \eqref{eq:delta_2 in terms of delta_1}, and we now check that there exists $\eps \in (0,1-\nicefrac{1}{q})$ such that this is fulfilled.
The ranges for $\delta_1$ then yield that
\begin{equation*}
    \delta_2 < \frac{1-\eps}{1+\eps}\Bigl(1-\frac{3}{p}\Bigr) - \frac{2}{1+\eps} \cdot 2 \eps \tand \delta_2 > \frac{1-\eps}{1+\eps}\Bigl(-\frac{3}{q}\Bigr) - \frac{2}{1+\eps} \cdot 2 \eps.
\end{equation*}
With regard to $\delta_2 \in (-3+\nicefrac{3}{q},-2+\nicefrac{3}{q})$, this requires that
\begin{equation}\label{eq:constraints from delta_2}
    -3+\frac{3}{q} < \frac{1-\eps}{1+\eps} \Bigl(1-\frac{3}{q}\Bigr) - \frac{2}{1+\eps} \cdot 2 \eps \tand -2+\frac{3}{q} > \frac{1-\eps}{1+\eps} \cdot \Bigl(-\frac{3}{q}\Bigr) - \frac{2}{1+\eps} \cdot 2 \eps.
\end{equation}
Straightforward calculations show that \eqref{eq:constraints from delta_2} can be guaranteed provided $\eps < 2- \nicefrac{3}{q}$ and $\eps > 1 - \nicefrac{3}{q}$.
By $\eps \in (0,1-\nicefrac{1}{q})$, we can find such $\eps$ provided $q > \nicefrac{3}{2}$, revealing that for such $q$, there exists $t \in (1,2-\nicefrac{1}{q})$ such that \eqref{eq:unweighted case}, is satisfied for $\delta_1$ and $\delta_2$ in the above ranges.

Subtracting the spatial average on $\cF_R = \cF_0 \cap B_R$, and denoting the associated solution operator by~$N_R$ by a slight abuse of notation, we conclude the assertion of the lemma below.
Note that we especially use classical properties of the trace to deduce that $N_R h \in \rL^q(\del\cS_0)$.

\begin{lem}\label{lem:lower reg Neumann problem}
Let $q \in (\nicefrac{3}{2},\infty)$.
Then there exists $t \in (1,2-\nicefrac{1}{q})$ such that for $h \in \rH^{1-t,q}(\cF_0)^3$, we have $N_R h \in \rL^q(\del \cS_0)^3$, and for sufficiently small $\eps > 0$, it holds that
\begin{equation*}
    \| N_R h \|_{\rW^{2-t,q}(\cF_0)} \le C \cdot \| h \|_{\rH^{1-t,q}(\cF_0)}.
\end{equation*}
\end{lem}

Equipped with the aforementioned lifting operators, we are now in the position to reformulate the inhomogeneous Stokes resolvent problem in terms of the lifting operators.

\begin{lem}\label{lem:equiv inhom Stokes resolvent problem ext}
Let $q \in (1,\infty)$ as well as $(f,\ell,\omega) \in \rL_\sigma^q(\cF_0) \times \C^3 \times \C^3$, and for $R > 0$ sufficiently large as specified above, $\cF_R = \cF_0 \cap B_R$.
Then $(u,p) \in \rW^{2,q}(\cF_0)^3 \times \widehat{\rW}^{1,q}(\cF_0) \cap \rL_0^q(\cF_R)$ satisfies
\begin{equation}\label{eq:inhom resolvent probl Stokes ext}
    \left\{
    \begin{aligned}
        \lambda u -  \Delta u + \nabla p 
        &= f, \enspace \mdiv u = 0, &&\tin \cF_0,\\
        u
        &= \ell + \omega \times y, &&\ton \del \cS_0
    \end{aligned}
    \right.
\end{equation}
if and only if
\begin{equation}\label{eq:formulae Pu Id - Pu and p Stokes ext}
    \left\{
    \begin{aligned}
        \lambda \cP u - A_0 \cP u + (A_0 - \lambda_0) \cP D_{\lambda_0,u}(\ell,\omega)
        &= \cP f,\\
        (\Id - \cP) u
        &= (\Id - \cP) D_{\lambda_0,u}(\ell,\omega),\\
        p
        &= N_R( \Delta \cP u) - \lambda N_R(\ell + \omega \times y).
    \end{aligned}
    \right.
\end{equation}
\end{lem}

\begin{proof}
First, we consider $(u,p) \in \rW^{2,q}(\cF_0)^3 \times \widehat{\rW}^{1,q}(\cF_0) \cap \rL_0^q(\cF_R)$ solving \eqref{eq:inhom resolvent probl Stokes ext}.
We then define
\begin{equation*}
    \tu \coloneqq u - D_{\lambda_0,u}(\ell,\omega) \tand \tp \coloneqq p - D_{\lambda_0,p}(\ell,\omega).
\end{equation*}
By construction, it follows that
\begin{equation*}
    \lambda \tu + \lambda D_{\lambda_0,u}(\ell,\omega) -  \Delta \tu + \nabla \tp = f + \lambda_0 D_{\lambda_0,u}(\ell,\omega), \enspace \mdiv \tu = 0, \tin \cF_0, \tand \tu = 0, \ton \del \cF_0.
\end{equation*}
Let us also observe that $\tu \in \rD(A_0)$, leading to $\cP \tu = \tu$.
An application of the Helmholtz projection to the preceding system then yields
\begin{equation*}
    \lambda \cP u - A_0 \cP u + (A_0 - \lambda_0) \cP D_{\lambda_0,u}(\ell,\omega) = \cP f.
\end{equation*}
Another consequence of $\cP \tu = \tu$ is that
\begin{equation*}
    (\Id - \cP)u = (\Id - \cP)(\tu + D_{\lambda_0,u}(\ell,\omega)) = (\Id - \cP) D_{\lambda_0,u}(\ell,\omega).
\end{equation*}
The representation of the Helmholtz projection from \autoref{lem:props Helmholtz and Stokes}(b) implies that $\Delta (\Id - \cP) u = 0$ in $\cF_0$ by $\mdiv u = 0$ in $\cF_0$.
Hence, we may reformulate \eqref{eq:inhom resolvent probl Stokes ext}$_1$ as
\begin{equation*}
    \lambda u - \Delta \cP u + \nabla p = f, \enspace \mdiv u = 0, \tin \cF_0.
\end{equation*}
Applying the divergence and the normal trace to this equation, which is legitimate thanks to \autoref{lem:props Helmholtz and Stokes}(a), and where we exploit $f \in \rL_\sigma^q(\cF_0)$ in conjunction with the characterization of this space in \autoref{lem:props Helmholtz and Stokes}(c), we find that $p$ solves
\begin{equation}\label{eq:weak Neumann problem pressure Stokes ext}
    \Delta p = 0, \tin \cF_0, \enspace \del_n p =  \Delta \cP u \cdot n - \lambda u \cdot n, \ton \del \cF_0.
\end{equation}
Note that $\mdiv \Delta \cP u = \mdiv u = 0$, so $ \Delta \cP u - \lambda u \in \rW_{\mdiv}^q(\cF_0)$, and $p = N_R ( \Delta \cP u - \lambda u) \in \widehat{\rW}^{1,q}(\cF_0) \cap \rL_0^q(\cF_R)$ is the unique solution to \eqref{eq:weak Neumann problem pressure Stokes ext}.
This shows that the representation formulae from \eqref{eq:formulae Pu Id - Pu and p Stokes ext} are valid.

On the other hand, suppose that $(u,p) \in \rW^{2,q}(\cF_0)^3 \times \widehat{\rW}^{1,q}(\cF_0) \cap \rL_0^q(\cF_R)$ fulfill \eqref{eq:formulae Pu Id - Pu and p Stokes ext}.
From \eqref{eq:formulae Pu Id - Pu and p Stokes ext}$_2$, we conclude that $\tu \coloneqq u - D_{\lambda_0,u}(\ell,\omega) \in \rD(A_0)$, so \eqref{eq:formulae Pu Id - Pu and p Stokes ext}$_1$ can be rephrased as
\begin{equation*}
    \cP(\lambda \tu - A_0 \tu) = \cP(f - (\lambda - \lambda_0)D_{\lambda_0,u}(\ell,\omega)).
\end{equation*}
By definition of the Helmholtz projection, there exists $\tp \in \widehat{\rW}^{1,q}(\cF_0)$, which we can additionally choose in $\rL_0^q(\cF_R)$ upon subtracting its spatial average in $\cF_R$ as in \eqref{eq:modified Neumann op ext}, and $(\tu,\tp)$ is the solution to
\begin{equation*}
    \lambda \tu -  \Delta \tu + \nabla \tp = f - (\lambda - \lambda_0)D_{\lambda_0,u}(\ell,\omega), \enspace \mdiv \tu = 0, \tin \cF_0, \enspace \tu = 0, \ton \del \cF_0.
\end{equation*}
As a result, for $D_{\lambda_0,u}$ and $D_{\lambda_0,p}$ emerging from \autoref{lem:sol to stat problem ext}, we find that $u = \tu + D_{\lambda_0,u}(\ell,\omega)$ and $p = \tp + D_{\lambda_0,p}(\ell,\omega)$ is a solution to \eqref{eq:inhom resolvent probl Stokes ext}, completing the proof of the equivalence.
\end{proof}

The preceding lemma allows us to rewrite the resolvent problem associated to \eqref{eq:lin ext}.
Indeed, for $\lambda \in \C$, $f_1 \in \rL^q(\cF_0)^3$ and $(f_2,f_3) \in \C^3 \times \C^3$, we consider
\begin{equation}\label{eq:resolvent problem ext}
    \left\{
    \begin{aligned}
        \lambda u -  \Delta u + \nabla p
        &= f_1, \enspace \mdiv u = 0, &&\tin \cF_0,\\
        u
        &= \ell + \omega \times y, &&\ton \del \cS_0,\\
        \lambda \mS \ell
        &= -\int_{\del \cS_0} \rT(u,p) n \rd \Gamma + f_2,\\
        \lambda J_0 \omega
        &= -\int_{\del \cS_0} y \times \rT(u,p) n \rd \Gamma + f_3.
    \end{aligned}
    \right.
\end{equation}
\autoref{lem:equiv inhom Stokes resolvent problem ext} together with the interface condition \eqref{eq:resolvent problem ext}$_2$ now enables us to rewrite \eqref{eq:resolvent problem ext}$_3$ and \eqref{eq:resolvent problem ext}$_4$ as
\begin{equation*}
    \begin{aligned}
        \lambda \mS \ell
        &= -2  \Bigl[\int_{\del \cS_0} \D(\cP u) n \rd \Gamma + \int_{\del \cS_0} \D((\Id - \cP)D_{\lambda_0,u}(\ell,\omega)) n \rd \Gamma\Bigr]\\
        &\quad + \int_{\del \cS_0} N_R( \Delta \cP u)n \rd \Gamma - \lambda \int_{\del \cS_0} N_R(\ell + \omega \times y) n \rd \Gamma + f_2,\\
        \lambda J_0 \omega
        &= -2  \Bigl[\int_{\del \cS_0} y \times \D(\cP u) n \rd \Gamma + \int_{\del \cS_0} y \times \D((\Id - \cP)D_{\lambda_0,u}(\ell,\omega)) n \rd \Gamma\Bigr]\\
        &\quad + \int_{\del \cS_0} y \times N_R( \Delta \cP u)n \rd \Gamma - \lambda \int_{\del \cS_0} y \times N_R(\ell + \omega \times y) n \rd \Gamma + f_3.
    \end{aligned}
\end{equation*}
The above two equations can in turn be reformulated as 
\begin{equation*}
    \lambda K \binom{\ell}{\omega} = C_1 \cP u + C_2 \binom{\ell}{\omega} + \binom{f_2}{f_3}, \twhere
\end{equation*}
\begin{equation}\label{eq:matrix K and ops C_1 and C_2}
    \begin{aligned}
        K
        &= \begin{pmatrix}
            \mS \Id_3 & 0\\
            0 & J_0
        \end{pmatrix} + M, \twith M \binom{\ell}{\omega} = \begin{pmatrix}
            \int_{\del \cS_0} N_R(\ell + \omega \times y)n \rd \Gamma\\
            \int_{\del \cS_0} y \times N_R(\ell + \omega \times y)n \rd \Gamma
        \end{pmatrix},\\
        C_1 \cP u
        &= \begin{pmatrix}
            -2 \int_{\del \cS_0} \D(\cP u) n \rd \Gamma + \int_{\del \cS_0} N_R(\Delta \cP u)n \rd \Gamma\\
            -2 \int_{\del \cS_0} y \times \D(\cP u) n \rd \Gamma + \int_{\del \cS_0} y \times N_R(\Delta \cP u)n \rd \Gamma
        \end{pmatrix}, \tand\\
        C_2 \binom{\ell}{\omega}
        &= \begin{pmatrix}
            -2\int_{\del \cS_0} \D((\Id_3 - \cP)D_{\lambda_0,u}(\ell,\omega)) n \rd \Gamma\\
            -2\int_{\del \cS_0} y \times \D((\Id_3 - \cP)D_{\lambda_0,u}(\ell,\omega)) n \rd \Gamma
        \end{pmatrix}.
    \end{aligned}
\end{equation}
The invertibility of the matrix $K$ follows from \cite[Lemma~4.3]{GGH:13}.
Recalling the Stokes operator $A_0$ in the exterior domain $\cF_0$ from \eqref{eq:Stokes op}, we are now in the position to define the fluid-structure operator in the present setting.
First, we introduce the ground space
\begin{equation}\label{eq:ground space fluid-structure op}
    \rX_0 = \{(u,\ell,\omega) \in \rL^q(\cF_0)^3 \times \C^3 \times \C^3 : u - \cP D_{\lambda_0,u}(\ell,\omega) \in \rL_\sigma^q(\cF_0)\}.
\end{equation}
From $\mdiv D_{\lambda_0,u}(\ell,\omega) = 0$ by construction of the lifting and the characterization of $\rL_\sigma^q(\cF_0)$ provided in \autoref{lem:props Helmholtz and Stokes}(c), it follows that $\rX_0$ from \eqref{eq:ground space fluid-structure op} coincides with $\rX_0$ as introduced in \autoref{sec:main results}.
On the other hand, let us introduce the space
\begin{equation*}
    \bX \coloneqq \{\Phi \in \rL_\sigma^q(\R^3) : D(\Phi) = 0 \tin \cS_0\}.
\end{equation*}
The following result can be found in \cite[Section~3.1]{EMT:23}.

\begin{lem}\label{lem:ground space isomorphic to closed subspace of Lq}
For the ground space $\rX_0$ from \eqref{eq:ground space fluid-structure op}, it holds that $\rX_0 \simeq \bX$.
In particular, $\rX_0$ is isomorphic to a closed subspace of $\rL^q$ on a $\sigma$-finite measure space.
\end{lem}

Next, we define the fluid-structure operator $\cA_\fs \colon \rD(\cA_\fs) \subset \rX_0 \to \rX_0$, where the domain takes the shape $\rD(\cA_\fs) = \{(\cP u,\ell,\omega) \in \rX_0 : \cP u - \cP D_{\lambda_0,u}(\ell,\omega) \in \rD(A_0)\}$, by
\begin{equation}\label{eq:fluid-structure op incompr Newtonian}
    \cA_\fs \coloneqq \begin{pmatrix}
        A_0 & (\lambda_0 - A_0) \cP D_{\lambda_0,u}\\
        K^{-1} C_1 & K^{-1} C_2
    \end{pmatrix}.
\end{equation}
One can show that $\rD(\cA_\fs)$ is identical with the space $\rX_1$ from \autoref{sec:main results}:
By construction of~$D_{\lambda_0,u}$ from~\eqref{eq:lifting op D_lambda} and its regularity properties stated in \eqref{eq:mapping props D_lambda}, the condition $\bigl(\cP u - \cP D_{\lambda_0,u}(\ell,\omega)\bigr) \in \rD(A_0)$ amounts to saying that $\cP u \in \rW^{2,q}(\cF_0)^3 \cap \rL_\sigma^q(\cF_0)$ as well as $u - (\ell + \omega \times y) = 0$ on $\del \cS_0$.
The equivalent reformulation of the resolvent problem \eqref{eq:resolvent problem ext} in terms of the fluid-structure operator is a direct consequence of \autoref{lem:equiv inhom Stokes resolvent problem ext} together with the definition of the fluid-structure operator and \eqref{eq:matrix K and ops C_1 and C_2}.

\begin{lem}\label{lem:equiv reform resolvent problem fluid-structure op ext}
Let $q \in (1,\infty)$ and terms on the right-hand side $(f_1,f_2,f_3) \in \rL_\sigma^q(\cF_0) \times \C^3 \times \C^3$.
Then $(u,p,\ell,\omega) \in \rW^{2,q}(\cF_0)^3 \times \widehat{\rW}^{1,q}(\cF_0) \cap \rL_0^q(\cF_R) \times \C^3 \times \C^3$ solves \eqref{eq:resolvent problem ext} if and only if
\begin{equation*}
    (\lambda \Id - \cA_\fs) \begin{pmatrix}
            \cP u\\ \ell\\ \omega
        \end{pmatrix} = \begin{pmatrix}
            \cP f_1\\ \Tilde{f_2}\\ \Tilde{f_3}
        \end{pmatrix}, \enspace (\Id - \cP)u = (\Id - \cP)D_{\lambda_0,u}(\ell,\omega) \tand p = N_R( \Delta \cP u) - \lambda N_R(\ell + \omega \times y),
\end{equation*}
with $(\Tilde{f_2},\Tilde{f_3})^\top = K^{-1} (f_2,f_3)^\top$.
\end{lem}

The next theorem provides the main result of this section on the bounded $\cH^\infty$-calculus of the fluid-structure operator up to a shift.
Note that the restriction on $q > \nicefrac{3}{2}$ arises from \autoref{lem:lower reg Neumann problem}.

\begin{thm}\label{thm:bdd H00-calculus of fluid-structure op incompr Newtonian}
For $q \in (\nicefrac{3}{2},\infty)$ and $\cA_\fs$ from \eqref{eq:fluid-structure op incompr Newtonian}, there exists $\lambda \ge 0$ such that $-\cA_\fs + \lambda \in \cH^\infty(\rX_0)$ with $\cH^\infty$-angle $\phi_{-\cA_\fs + \lambda}^\infty = 0$.
\end{thm}

\begin{proof}
The main difficulty is that the domain of the fluid-structure operator $\cA_\fs$ from \eqref{eq:fluid-structure op incompr Newtonian} is {\em non-diagonal}, while the theory recalled in \autoref{sec:Hoo, SMR & bilin SPDEs} only allows for operators with diagonal domain.
Hence, we employ a decoupling argument.
More precisely, for 
\begin{equation}\label{eq:Y_0}
    \rY_0 \coloneqq \rL_\sigma^q(\cF_0) \times \C^3 \times \C^3,
\end{equation}
we prove that there exist $S \in \cL(\rX_0,\rY_0)$, $S^{-1} \in \cL(\rY_0,\rX_0)$ such that $\tcA_\fs = S \cA_\fs S^{-1}$, with diagonal domain $\rD(\tcA_\fs) = S \rD(\cA_\fs)$.
Then we show that $\tcA_\fs$ admits a bounded $\Hinfty$-calculus by means of the theory of block operator matrices and perturbation theory.
Finally, we deduce the boundedness of the $\Hinfty$-calculus of $\cA_\fs$ by invoking \autoref{lem:preservation of results under sim trafe}.

In the following, we consider
\begin{equation}\label{eq:decoupling matrices}
    S = \begin{pmatrix}
        \Id & - \cP D_{\lambda_0,u}\\
        0 & \Id
    \end{pmatrix} \tand S^{-1} = \begin{pmatrix}
        \Id & \cP D_{\lambda_0,u}\\
        0 & \Id
    \end{pmatrix}.
\end{equation}
From \eqref{eq:mapping props D_lambda}, we recall $D_{\lambda_0,u} \in \cL(\C^3 \times \C^3,\rW^{2,q}(\cF_0)^3)$, so we find $S \in \cL(\rX_0,\rY_0)$, $S^{-1} \in \cL(\rY_0,\rX_0)$ and
\begin{equation*}
    \rY_0 = S \rX_0 = \{(\cP \Tilde{u},\ell,\omega) \in \rL_\sigma^q(\cF_0) \times \C^3 \times \C^3\}.
\end{equation*}
Moreover, we calculate
\begin{equation}\label{eq:similarity trafo fluid-structure op}
    \tcA_\fs = S \cA_\fs S^{-1} = \begin{pmatrix}
        A_0 - \cP D_{\lambda_0,u} K^{-1} C_1 & \lambda_0 \cP D_{\lambda_0,u} - \cP D_{\lambda_0,u} (K^{-1} C_1 \cP D_{\lambda_0,u} + K^{-1} C_2)\\
        K^{-1} C_1 & K^{-1} C_1 \cP D_{\lambda_0,u} + K^{-1} C_2
    \end{pmatrix},
\end{equation}
and this operator has diagonal domain $\rD(\tcA_\fs) = \rW^{2,q}(\cF_0)^3 \cap \rW_0^{1,q}(\cF_0)^3 \cap \rL_\sigma^q(\cF_0) \times \C^3 \times \C^3$.

The disadvantage of the decoupled operator matrix is its complicated shape.
In order to handle this, we use the decomposition into $\tcA_\fs = \tcA_{\fs,\rI} + \cB_\fs$, where
\begin{equation*}
    \tcA_{\fs,\rI} = \begin{pmatrix}
        A_0 & \lambda_0 \cP D_{\lambda_0,u}\\
        K^{-1} C_1 & K^{-1} C_2
    \end{pmatrix} \tand \cB_\fs = \begin{pmatrix}
        - \cP D_{\lambda_0,u} K^{-1} C_1 & -\cP D_{\lambda_0,u} (K^{-1} C_1 \cP D_{\lambda_0,u} + K^{-1} C_2)\\
        0 & K^{-1} C_1 \cP D_{\lambda_0,u}
    \end{pmatrix}.
\end{equation*}

Next, we show that $K^{-1} C_2 \in \cL(\C^3 \times \C^3)$.
By the invertibility of $K$, this task reduces to verifying that $C_2 \in \cL(\C^3 \times \C^3)$.
To this end, we recall from \eqref{eq:mapping props D_lambda} that $D_{\lambda_0,u} \in \cL(\C^3 \times \C^3,\rW^{2,q}(\cF_0)^3)$.
Let us observe that the application of $\Id_3 - \cP$ does not affect the regularity, so $(\Id_3 - \cP) D_{\lambda_0,u}(\ell,\omega) \in \rW^{2,q}(\cF_0)^3$ and thus $\D((\Id_3 - \cP) D_{\lambda_0,u}(\ell,\omega)) \in \rW^{1,q}(\cF_0)^{3 \times 3}$.
Standard trace theory then allows us to conclude $\D((\Id_3 - \cP) D_{\lambda_0,u}(\ell,\omega)) \in \rW^{1-\nicefrac{1}{q},q}(\del\cF_0)^{3 \times 3}$.
In total, this reveals that $C_2(\ell,\omega) \in \C^3 \times \C^3$, or, in other words, $C_2 \in \cL(\C^3 \times \C^3)$.

In view of $-A_0 \in \Hinfty(\rL_\sigma^q(\cF_0))$ with $\phi_{-A_0}^\infty = 0$ by \autoref{lem:props Helmholtz and Stokes}(d), it remains to show the diagonal dominance of $\tcA_{\fs,\rI}$.
By the boundedness of $K^{-1} C_2$, we also require $\lambda_0 \cP D_{\lambda_0,u} \in \cL(\C^3 \times \C^3,\rL_\sigma^q(\cF_0))$.
This is implied by \eqref{eq:mapping props D_lambda}.
By \autoref{def:diag dom op matrix}, the last task for the diagonal dominance of $\tcA_{\fs,\rI}$ is to establish that~$K^{-1} C_1$ is relatively $A_0$-bounded.
Thanks to the invertibility of $K$, it is enough to show this property for~$C_1$.
For this purpose, let $\cP u \in \rD(A_0)$.
From standard trace theory, \autoref{lem:reg of Neumann op} and the continuity of~$N_R$ from $\rW_{\mdiv}^q(\cF_0)$ to $\widehat{\rW}^{1,q}(\cF_0) \cap \rL_0^q(\cF_R)$, we conclude that 
\begin{equation}\label{eq:first est of C_1}
    \| C_1 \cP u \|_{\C^3 \times \C^3} \le C \cdot \| \cP u \|_{\rW^{2,q}(\cF_0)} \le C\bigl(\| A_0 \cP u \|_{\rL_\sigma^q(\cF_0)} + \| \cP u \|_{\rL_\sigma^q(\cF_0)}\bigr).
\end{equation}
An application of \autoref{prop:op theoret props of block op matrices} then yields the existence of $\lambda \ge 0$ such that $-\tcA_{\fs,\rI} + \lambda \in \Hinfty(\rY_0)$ as well as~$\phi_{-\tcA_{\fs,\rI} + \lambda}^\infty = 0$. 

For the boundedness of the $\Hinfty$-calculus of $\tcA_\fs$, we need to treat $\cB_\fs$ as a perturbation as in \autoref{prop:pert of Hinfty-calculus}.
To this end, for $t \in (1,2-\nicefrac{1}{q})$ coming from \autoref{lem:lower reg Neumann problem}, let $\alpha = \frac{3}{2} - \frac{t}{2}$.
Note that this choice yields $\alpha \in (\nicefrac{1}{2}+\nicefrac{1}{2q},1)$. 
From the previous part of the proof together with \autoref{lem:rel of Hinfty with other concepts}(c) and well-known interpolation theory, we find that
\begin{equation*}
    \rD((-\tcA_{\fs,\rI}+\lambda)^\alpha) \cong \rH^{2 \alpha,q}(\cF_0)^3 \cap \rW_0^{1,q}(\cF_0)^3 \cap \rL_\sigma^q(\cF_0) \times \C^3 \times \C^3,
\end{equation*}
where $\lambda \ge 0$ is such that $-\tcA_{\fs,\rI} + \lambda \in \Hinfty(\rY_0)$.
For such $\alpha$, let $(\cP u,\ell,\omega) \in \rD((-\tcA_{\fs,\rI}+\lambda)^\alpha)$.
The goal now is to establish the relative boundedness of the operator $\cB_\fs$ with respect to a fractional power of~$\tcA_{\fs,\rI}$.
We start by discussing the operator $C_1$ from \eqref{eq:matrix K and ops C_1 and C_2}, and we further split it into
\begin{equation*}
    C_1 \cP u = C_{1,1} \cP u + C_{1,2} \cP u \coloneqq \begin{pmatrix}
            -2 \int_{\del \cS_0} \D(\cP u) n \rd \Gamma\\
            -2 \int_{\del \cS_0} y \times \D(\cP u) n \rd \Gamma
        \end{pmatrix} + \begin{pmatrix}
            \int_{\del \cS_0} N_R(\Delta \cP u)n \rd \Gamma\\
            \int_{\del \cS_0} y \times N_R(\Delta \cP u)n \rd \Gamma
        \end{pmatrix}.
\end{equation*}
For the first term, we use usual trace theory and $\rD((-A_0)^\alpha) \cong \rH^{2\alpha,q}(\cF_0)^3 \cap \rL_\sigma^q(\cF_0) \hookrightarrow \rW^{1+\nicefrac{1}{q}+\eps,q}(\cF_0)^3$, for sufficiently small $\eps > 0$, to deduce that
\begin{equation*}
    | C_{1,1} \cP u | \le C \cdot \| \cP u \|_{\rW^{1+\nicefrac{1}{q}+\eps,q}(\cF_0)} \le C \cdot \| (-A_0)^\alpha \cP u \|_{\rL_\sigma^q(\cF_0)}.
\end{equation*}
On the other hand, for $C_{1,2}$, we invoke \autoref{lem:lower reg Neumann problem} for the choice $t = 3 - 2 \alpha$, where we recall that $2 \alpha - 1 > \nicefrac{1}{q}$, together with mapping properties of the Laplacian and the aforementioned fractional powers of the Stokes operator to infer that
\begin{equation*}
    |C_{1,2} \cP u| \le C \cdot \| N_R(\Delta \cP u) \|_{\rW^{\nicefrac{1}{q}+\eps,q}(\cF_0)} \le C \cdot \| \cP u \|_{\rH^{2\alpha}(\cF_0)} \le C \cdot \| (-A_0)^\alpha \cP u \|_{\rL_\sigma^q(\cF_0)}.
\end{equation*}
This reveals that $|C_1 \cP u| \le C \cdot  \| (-A_0)^\alpha \cP u \|_{\rL_\sigma^q(\cF_0)}$.
Additionally exploiting the mapping properties of~$D_{\lambda_0,u}$ from \eqref{eq:mapping props D_lambda}, we thus find that
\begin{equation}\label{eq:pert left top entry}
    \| \cP D_{\lambda_0,u} K^{-1} C_1 \cP u \|_{\rL_\sigma^q(\cF_0)} \le C \cdot |C_1 \cP u| \le C \cdot \| (-A_0)^\alpha \cP u \|_{\rL_\sigma^q(\cF_0)}.
\end{equation}
Next, using the first estimate from \eqref{eq:first est of C_1} in conjunction with \eqref{eq:mapping props D_lambda}, we argue that
\begin{equation}\label{eq:pert right bottom entry}
    \| K^{-1} C_1 \cP D_{\lambda_0,u}(\ell,\omega) \|_{\C^3 \times \C^3} \le C \cdot \| \cP D_{\lambda_0,u}(\ell,\omega) \|_{\rW^{2,q}(\cF_0)} \le C \cdot \| (\ell,\omega) \|_{\C^3 \times \C^3}.
\end{equation}
Combining \eqref{eq:pert right bottom entry} with $C_2 \in \cL(\C^3 \times \C^3)$ and \eqref{eq:mapping props D_lambda}, we obtain
\begin{equation}\label{eq:pert right top entry}
    \| \lambda_0 \cP D_{\lambda_0,u} (K^{-1} C_1 \cP D_{\lambda_0,u}(\ell,\omega) + K^{-1} C_2(\ell,\omega)) \|_{\rL_\sigma^q(\cF_0)} \le C \cdot \| (\ell,\omega) \|_{\C^3 \times \C^3}.
\end{equation}
For $\alpha = (\frac{3}{2} - \frac{t}{2}) \in (\nicefrac{1}{2}+\nicefrac{1}{2q},1)$, a concatenation of \eqref{eq:pert left top entry}, \eqref{eq:pert right bottom entry}, and \eqref{eq:pert right top entry} leads to the existence of~$C > 0$ so that
\begin{equation*}
    \| \cB_\fs (\cP u,\ell,\omega) \|_{\rY_0} \le C\bigl(\| (-\tcA_{\fs,\rI} + \lambda)^\alpha (\cP u,\ell,\omega) \|_{\rY_0} + \| (\cP u,\ell,\omega) \|_{\rY_0}\bigr), \tfor (\cP u,\ell,\omega) \in \rD((-\tcA_{\fs,\rI}+\lambda)^\alpha).
\end{equation*}
Let us observe that \eqref{eq:pert left top entry}, \eqref{eq:pert right bottom entry}, and \eqref{eq:pert right top entry} can also be used to show that $-(\tcA_{\fs,\rI} + \cB_\fs) + \lambda$ is sectorial and invertible with spectral angle equal to zero for $\lambda \ge 0$ sufficiently large thanks to standard perturbation theory of sectorial operators, see, e.g., \cite[Corollary~3.1.6]{PS:16}.
\autoref{prop:pert of Hinfty-calculus} then implies that the operator $\tcA_\fs = \tcA_{\fs,\rI} + \cB_\fs$ satisfies $-\tcA_\fs + \lambda \in \Hinfty(\rY_0)$, with $\phi_{-\tcA_\fs + \lambda}^\infty = 0$, for $\lambda \ge 0$ sufficiently large.
The assertion of the theorem is then a consequence of \autoref{lem:preservation of results under sim trafe} applied to the similarity transform from~\eqref{eq:similarity trafo fluid-structure op}.
\end{proof}

Next, we state several interesting consequences of \autoref{thm:bdd H00-calculus of fluid-structure op incompr Newtonian}, including the stochastic maximal regularity in~(c).
The assertions in~(a) and~(b) follow from the respective assertions in \autoref{lem:rel of Hinfty with other concepts} upon noting that $\rX_0$ is a UMD space, since it is isomorphic to a closed subspace of $\rL^q$ on a $\sigma$-finite measure space by \autoref{lem:ground space isomorphic to closed subspace of Lq}.
The latter observation also yields the assertion of~(c) by \autoref{prop:stoch max reg via Hinfty subset of Lq}.
Let us observe that the assertion in~(a) can also be obtained for $q \in (1,\nicefrac{3}{2}]$ upon slightly modifying the arguments in the proof of \autoref{thm:bdd H00-calculus of fluid-structure op incompr Newtonian} and making use of \autoref{lem:reg of Neumann op} instead of \autoref{lem:lower reg Neumann problem}.
This is possible by virtue of the simpler perturbation theory for maximal $\rL^p$-regularity characterized in terms of $\cR$-sectoriality in the situation of UMD spaces.
We omit the details here.

\begin{cor}\label{cor:cons of H00-calc of fs op incompr Newtonian}
Consider $q \in (\nicefrac{3}{2},\infty)$ and $\lambda \ge 0$ as in \autoref{thm:bdd H00-calculus of fluid-structure op incompr Newtonian}.
Then the operator $\cA_\fs$ from \eqref{eq:fluid-structure op incompr Newtonian} satisfies the following properties.
\begin{enumerate}[(a)]
    \item The operator $\cA_\fs - \lambda$ generates a bounded analytic semigroup $(\mre^{t (\cA_\fs - \lambda)})_{t \ge 0}$ of angle $\nicefrac{\pi}{2}$ on $\rX_0$.
    In particular, $\Afs$ generates a (strongly continuous) analytic semigroup on $\rX_0$.
    Furthermore, $-\cA_\fs + \lambda$ has maximal $\rL^p$-regularity on $\rX_0$, i.e., for all $(\cP f_1,f_2,f_3) \in \rL^p(\R_+;\rX_0)$, there is unique $(\cP u,\ell,\omega) \in \rW^{1,p}(\R_+;\rX_0) \cap \rL^p(\R_+;\rD(\cA_\fs))$ solving \eqref{eq:lin ext} with $u_0 = \ell_0 = \omega_0 = 0$.
    \item For $\alpha \in (0,1)$ and $\rX_\alpha$ related to $-\cA_\fs + \lambda$ as revealed in \autoref{lem:rel of Hinfty with other concepts}, we have $\rX_\alpha \cong [\rX_0,\rX_{\cA_\fs}]_\alpha$.
    \item Let $p$, $q \in [2,\infty)$ and $\kappa \in [0,\nicefrac{p}{2} - 1) \cup \{0\}$.
    Then $(-\Afs,0) \in \cSMR_{p,\kappa}^\bullet$ as in \autoref{def:stoch max reg}(b).
\end{enumerate}
\end{cor}

\section{Proof of the main results}\label{sec:proof main results}

This section is dedicated to showing the local strong well-posedness results \autoref{thm:loc strong wp stoch fsi} and \autoref{thm:loc strong wp stoch fsi Hilbert space} and the blow-up criteria \autoref{thm:blow-up crit stoch fsi} and \autoref{thm:blow-up crit stoch fsi Hilbert space}.
The essential idea is to reformulate~\eqref{eq:stoch forced fsi problem intro} in terms of the fluid-structure operator $\cA_\fs$ from \eqref{eq:fluid-structure op incompr Newtonian}, and to show that the resulting problem fits into the framework presented in \autoref{ssec:stoch max reg and semilin SPDEs} thanks to the stochastic maximal regularity as well as nonlinear estimates. 

First, we deal with the interpolation spaces.
More precisely, we prove \autoref{lem:interpol spaces}.
In this context, there are two main obstacles, namely, the non-diagonal domain of the fluid-structure operator, and that the Helmholtz projection generally does not preserve Dirichlet boundary conditions.

\begin{proof}[Proof of~\autoref{lem:interpol spaces}]
We start by discussing how to handle the coupling condition in the function spaces.
For this purpose, we invoke the decoupling argument from the proof of \autoref{thm:bdd H00-calculus of fluid-structure op incompr Newtonian}.
The (non-diagonal) domain $\rX_1 = \rD(\Afs)$ of the fluid-structure operator $\Afs$ and the (diagonal) domain of the decoupled fluid-structure operator $\tcA_\fs$ are linked via $\rD(\tcA_\fs) = S \rD(\Afs)$, where $S$ has been made precise in \eqref{eq:decoupling matrices}.
We also recall the space $\rY_0$ from \eqref{eq:Y_0} and set $\rY_1 \coloneqq \rD(\tcA_\fs) = \rW^{2,q}(\cF_0)^3 \cap \rW_0^{1,q}(\cF_0)^3 \cap \rL_\sigma^q(\cF_0) \times \C^3 \times \C^3$.

Arguing in a similar way as in \cite[Lemma~5.3]{BBH:23}, we find that the interpolation spaces in the coupled setting can be recovered from those in the decoupled setting via the matrix $S^{-1}$ by
\begin{equation*}
    [\rX_0,\rX_1]_\theta = S^{-1}([\rY_0,\rY_1]_\theta) \tand (\rX_0,\rX_1)_{\theta,p} = S^{-1}((\rY_0,\rY_1)_{\theta,p}), \tfor \theta \in (0,1) \tand p \in (1,\infty).
\end{equation*}
Below, for $i \in \{1,2\}$, we will also write $\rY_i = \rY_i^u \times \rY_i^\ell \times \rY_i^\omega$ for the decoupled spaces.
We then get the identity $[\rY_0,\rY_1]_\theta = \left\{(\cP u,\ell,\omega) \in [\rY_0^u,\rY_1^u]_\theta \times \C^3 \times \C^3\right\}$.
This results in
\begin{equation}\label{eq:coupled complex interpol}
    [\rX_0,\rX_1]_\theta = S^{-1} ([\rY_0,\rY_1]_\theta) = \left\{(\cP u,\ell,\omega) \in \rY_0 : \cP u - \cP D_{\lambda_0,u}(\ell,\omega) \in [\rY_0^u,\rY_1^u]_\theta\right\}, \tfor \theta \in (0,1).
\end{equation}
A similar argument as above for the shape of $\rX_1$ shows that the space $\rX_\beta$ includes the interface condition provided $\beta$ is sufficiently large so that elements in $[\rY_0^u,\rY_1^u]_\beta$ admit a trace, which is the case for $\beta > \nicefrac{1}{2q}$.
Likewise, for the real interpolation spaces with $\theta \in (0,1)$ and $p \in (1,\infty)$, it follows that
\begin{equation}\label{eq:coupled real interpol}
    (\rX_0,\rX_1)_{\theta,p} = S^{-1}((\rY_0,\rY_1)_{\theta,p}) = \left\{(\cP u,\ell,\omega) \in \rY_0 : \cP u - \cP D_{\lambda_0,u}(\ell,\omega) \in (\rY_0^u,\rY_1^u)_{\theta,p}\right\},
\end{equation}
where the interface condition enters if $\theta > \nicefrac{1}{2q}$.
Note that for $\theta < \nicefrac{1}{2q}$, we still have $u \cdot n = (\ell + \omega \times y) \cdot n$ by the characterization of the space $\rL_\sigma^q(\cF_0)$ from \autoref{lem:props Helmholtz and Stokes}(c).
In summary, the computation of the interpolation spaces reduces to the study of the non-coupled ones related to $\rY_0$ and $\rY_1$.

Concerning this reduced problem, the main challenge is to find a projection $Q$ on $\rL^q(\cF_0)^3$ such that for the domain $\rD(\DeltaD) \coloneqq \rW^{2,q}(\cF_0)^3 \cap \rW_0^{1,q}(\cF_0)^3$ of the Dirichlet Laplacian $\DeltaD u \coloneqq \Delta u$ for $u \in \rD(\DeltaD)$, it holds that $Q \rD(\DeltaD) = \rD(\DeltaD) \cap \rR(Q)$.
Let us observe that this can be done in a similar manner as in \cite[Section~4]{PW:18}.
For the (shifted) Dirichlet Laplacian $\DeltaD^1 \coloneqq \Id - \DeltaD$ with $\rD(\DeltaD^1) = \rW^{2,q}(\cF_0)^3 \cap \rW_0^{1,q}(\cF_0)^3$ and the shifted negative Stokes operator $A_1 \coloneqq \Id - A_0$ with $\rD(A_1) = \rW^{2,q}(\cF_0)^3 \cap \rW_0^{1,q}(\cF_0)^3 \cap \rL_\sigma^q(\cF_0)$.
The projection $Q \colon \rD(\DeltaD^1) = \rD(\DeltaD) \to \rD(A_0) = \rD(A_1)$ is then given by $Q u \coloneqq A_1^{-1} \cP \DeltaD^1 u$.

One can argue that $Q$ admits a unique extension $\tQ \in \cL(\rL^q(\cF_0)^3,\rL_\sigma^q(\cF_0))$ that is also a projection such that $\left. \tQ \right|_{\rL_\sigma^q(\cF_0)} = \left. \Id \right|_{\rL_\sigma^q(\cF_0)}$.
By $\tQ \rD(\DeltaD) = \rD(\DeltaD) \cap \rR(\tQ) = \rD(A_0)$, we get $\rL^q(\cF_0)^3 = \rL_\sigma^q(\cF_0) \oplus \rN(\tQ)$ and $\rD(\DeltaD) = [\rD(\DeltaD) \cap \rR(\tQ)] \oplus [\rD(\DeltaD) \cap \rN(\tQ)]$.
The well-known boundedness of the imaginary powers of $\DeltaD^1$ together with interpolation theoretic arguments, see \cite[Theorem~1.17.1.1]{Tri:78}, then yields
\begin{equation*}
    [\rL_\sigma^q(\cF_0),\rD(A_0)]_\theta = [\tQ \rL^q(\cF_0)^3,\tQ \rD(\DeltaD^1)]_\theta = \tQ [\rL^q(\cF_0)^3,\rD(\DeltaD^1)]_\theta = \rD((\DeltaD^1)^\theta) \cap \rR(\tQ),
\end{equation*}
for $\theta \in (0,1)$, and likewise for the real interpolation spaces.
In view of the reduction argument in the first part of the proof, this shows the assertion of the lemma.
\end{proof}

Next, we rewrite the stochastic fluid-rigid body interaction problem \eqref{eq:stoch forced fsi problem intro} as a bilinear SPDE of the form as in \eqref{eq:bilin SPDE}.
In fact, recalling the fluid-structure operator $\Afs$ from \eqref{eq:fluid-structure op incompr Newtonian}, similarly as in \eqref{lem:equiv reform resolvent problem fluid-structure op ext}, for~$w = (\cP v, \ell,\omega)$, we find that \eqref{eq:stoch forced fsi problem intro} is equivalent with
\begin{equation}\label{eq:stoch fsi as bilin SPDE}
    \left\{
    \begin{aligned}
        \rd w - \Afs w \rd t 
        &= F(w) \rd t + (B w + f) \rd W_t^n,\\
        w(0)
        &= \begin{pmatrix}
            \cP v_0\\ \ell_0\\ \omega_0
        \end{pmatrix},
    \end{aligned}
    \right.
\end{equation}
where 
\begin{equation}\label{eq:terms on RHS}
    F(w) = \begin{pmatrix}
        -\cP[((v-\ell) \cdot \nabla)v]\\ 0\\ 0
    \end{pmatrix}, \enspace (B w) \mre_n = \begin{pmatrix}
        \cP[(b_n \cdot \nabla)v]\\ 0\\ 0
    \end{pmatrix} \tand f \mre_n = \begin{pmatrix}
            \cP f_v^n\\ \Tilde{f}_\ell^n\\ \Tilde{f}_\omega^n
        \end{pmatrix},
\end{equation}
with $\binom{\Tilde{f}_\ell^n}{\Tilde{f}_\omega^n} = K^{-1} \binom{f_\ell^n}{f_\omega^n}$, and $K$ has been made precise in \eqref{eq:matrix K and ops C_1 and C_2}.

At this stage, we begin with the verification that \eqref{eq:stoch fsi as bilin SPDE} lies within the scope of \autoref{ssec:stoch max reg and semilin SPDEs}.
First, for the ground space $\rX_0$ as made precise in \eqref{eq:ground space fluid-structure op}, we recall from \autoref{lem:ground space isomorphic to closed subspace of Lq} that it is isomorphic to a closed subspace of $\rL^q$ on a $\sigma$-finite measure space, so for $q \ge 2$, we especially deduce that $\rX_0$ is a UMD space with type $2$, see also \cite[Chapter~7]{HvNVW:17} for more details on this property.

After establishing this basic property, we next show that $(-\Afs,B) \in \cSMR_{p,\kappa}^\bullet$, where the fluid-structure operator~$\Afs$ has been introduced in \eqref{eq:fluid-structure op incompr Newtonian}, and $B$ is as made precise in \eqref{eq:terms on RHS} indeed admits stochastic maximal regularity when assuming the transport noise term to be sufficiently small.

\begin{lem}\label{lem:stoch max reg of lin part}
Let $p$, $q \in [2,\infty)$, $\kappa \in [0,\nicefrac{p}{2}-1) \cup \{0\}$, and suppose that there exists $\eps > 0$ such that~$\| (b_n)_{n \ge 1} \|_{\rW^{1,\infty}(\cF_0;\ell^2)} < \eps$.
If $\eps > 0$ is sufficiently small, it holds that $(-\Afs,B) \in \cSMR_{p,\kappa}^\bullet$.
\end{lem}

\begin{proof}
First, from \autoref{cor:cons of H00-calc of fs op incompr Newtonian}, we recall that $(-\Afs,0) \in \cSMR_{p,\kappa}^\bullet$.
In addition, it also readily follows that $(-\Afs,0) \in \cSMR_p$.
In the following, we employ the notation $(B w) \mre_n = (\cP[(b_n \cdot \nabla)v],0,0)^\top$.
Invoking the shape of $\rX_{\nicefrac{1}{2}}$ from \eqref{eq:X_1/2}, recalling from \autoref{ssec:prelims} that $\gamma(\ell^2,\rW^{1,q}(\cF_0)) = \rW^{1,q}(\cF_0;\ell^2)$, and making use of the assumption, for $w \in \rX_1$, we infer that
\begin{equation*}
    \begin{aligned}
        \| B w \|_{\gamma(\ell^2,\rX_{\nicefrac{1}{2}})}
        &\le C \cdot \left\| \sum_{n \ge 1} [(b_n \cdot \nabla)v] \right\|_{\gamma(\ell^2,\rW^{1,q}(\cF_0))}
        \le C \cdot \left\| \sum_{n \ge 1} [(b_n \cdot \nabla)v] \right\|_{\rW^{1,q}(\cF_0;\ell^2)}\\
        &\le \| (b_n)_{n \ge 1} \|_{\rW^{1,\infty}(\cF_0;\ell^2)} \cdot \| v \|_{\rW^{2,q}(\cF_0)}
        \le C \eps \cdot \| w \|_{\rX_1}.
    \end{aligned}
\end{equation*}
Hence, the map $B$ satisfies the desired mapping properties from \autoref{lem:pert SMR}, and we deduce from the latter lemma that indeed, $(-\Afs,B) \in \cSMR_{p,\kappa}^\bullet$.
\end{proof}

In the lemma below, we discuss the estimates of the bilinear term $F$ from \eqref{eq:terms on RHS}.

\begin{lem}\label{lem:nonlin est}
Recall $F(w) = G(w,w)$ from \eqref{eq:terms on RHS} and the shape of $\rX_\beta = [\rX_0,\rX_1]_\beta$ from \autoref{lem:interpol spaces}.
\begin{enumerate}[(a)]
    \item Let $q \in (1,3)$.
    Then for $\beta \ge \frac{1}{4}(1+\frac{3}{q}) \in (0,1)$, the map $G \colon \rX_\beta \times \rX_\beta \to \rX_0$ is bounded.
    \item If $q \ge 3$, then for $\beta \in [\frac{1}{2},1)$, the map $G \colon \rX_\beta \times \rX_\beta \to \rX_0$ is bounded.
\end{enumerate}    
\end{lem}

\begin{proof}
Let us start with the case $q \in (1,3)$, which is the critical one.
H\"older's inequality first yields that
\begin{equation*}
    \| G(w,w) \|_{\rX_0} \le \| ((v-\ell) \cdot \nabla) v \|_{\rL^q(\cF_0)} \le C \bigl(\| v \|_{\rL^{qr'}(\cF_0)} \cdot \| v \|_{\rH^{1,qr}(\cF_0)} + |\ell| \cdot \| v \|_{\rW^{1,q}(\cF_0)}\bigr),
\end{equation*}
where $r$, $r' \in (1,\infty)$ are such that $\nicefrac{1}{r} + \nicefrac{1}{r'} = 1$.
With the choice $\nicefrac{3}{qr} = \nicefrac{1}{2}(1+\nicefrac{3}{q})$, for which we require that $q < 3$, we can ensure that the Sobolev indices of $\rL^{qr'}(\cF_0)$ and $\rH^{1,qr}(\cF_0)$ coincide.
Sobolev embeddings, see, for example, \cite[Theorem~4.6.1]{Tri:78}, yield that
\begin{equation*}
    \rH^{2\beta,q}(\cF_0) \hookrightarrow \rL^{qr'}(\cF_0) \cap \rH^{1,qr}(\cF_0) \tif \beta \ge \frac{1}{4}\left(1+\frac{3}{q}\right).
\end{equation*}
Let us observe that the preceding condition on $\beta$ together with $q \in (1,3)$ already implies that $\beta \in (\nicefrac{1}{2},1)$.
It then follows from Sobolev embeddings that $\rH^{2 \beta,q}(\cF_0) \hookrightarrow \rW^{1,q}(\cF_0)$.
Putting together these arguments, and invoking the above shape of the space $\rX_\beta$, we find that
\begin{equation*}
     \| G(w,w) \|_{\rX_0} \le C \cdot \| w \|_{\rX_\beta}^2, \tfor \beta \ge \frac{1}{4}\left(1+\frac{3}{q}\right) \tand q < 3.
\end{equation*}
This shows the assertion of~(a).

For~(b), we first observe that for all $r \in (1,\infty)$, we obtain that $1-\frac{3}{qr} \ge -\frac{3}{q r'}$, since $r (3 + q) > 6$ in this case.
Hence, the embedding $\rH^{2 \beta,q}(\cF_0) \hookrightarrow \rH^{1,qr}(\cF_0)$ is always the more restrictive one in this situation.
Thus, this time, we estimate
\begin{equation*}
    \| G(w,w) \|_{\rX_0} \le \| ((v-\ell) \cdot \nabla) v \|_{\rL^q(\cF_0)} \le C \bigl(\| v \|_{\rL^{\infty}(\cF_0)} \cdot \| v \|_{\rW^{1,q}(\cF_0)} + |\ell| \cdot \| v \|_{\rW^{1,q}(\cF_0)}\bigr).
\end{equation*}
It readily follows that $\rH^{2 \beta,q}(\cF_0) \hookrightarrow \rW^{1,q}(\cF_0)$ provided $\beta \ge \frac{1}{2}$, so by \autoref{lem:interpol spaces}, we get
\begin{equation*}
     \| G(w,w) \|_{\rX_0} \le C \cdot \| w \|_{\rX_\beta}^2, \tfor \beta \in \Bigl(\frac{1}{2},1\Bigr) \tand q \ge 3. \qedhere
\end{equation*}
\end{proof}

We are now in the position to prove the local strong well-posedness result \autoref{thm:loc strong wp stoch fsi}.

\begin{proof}[Proof of \autoref{thm:loc strong wp stoch fsi}]
For an application of \autoref{prop:loc ex result bilin SPDE}, it remains to verify the validity of \autoref{ass:main ass bilin setting} in the present setting.
Note that we have already discussed that $\rX_0$ is a UMD space with type~$2$ provided $q \ge 2$, and that $\lambda - \Afs \in \Hinfty(\rX_0)$, see \autoref{thm:bdd H00-calculus of fluid-structure op incompr Newtonian}.
In addition, we have already established that $(-\Afs,B) \in \cSMR_{p,\kappa}^\bullet$ in \autoref{lem:stoch max reg of lin part}.
As a consequence, it remains to show that there are $p \in [2,\infty)$, $\kappa \in [0,\nicefrac{p}{2}-1) \cup \{0\}$ and $\beta \in (1-\frac{1+\kappa}{p},1)$ such that $\frac{1+\kappa}{p} \le 2(1-\beta)$, and the map $F(w) = G(w,w)$ as introduced in \eqref{eq:terms on RHS} is bilinear and bounded from $\rX_\beta \times \rX_\beta$ to $\rX_0$.

In the following, we distinguish between the two cases and start with the critical one.
Hence, we first consider $q \in [2,3)$, $p \in (2,\infty)$ and $\beta \ge \frac{1}{4}(1+\frac{3}{q})$ as in \autoref{lem:nonlin est}.
Note that $\beta < 1$ is directly implied by~$q \ge 2 > 1$.
On the other hand, the condition $\beta > 1 - \frac{1+\kappa}{p}$ is equivalent with $\kappa > \frac{3p}{4}(1-\frac{1}{q}) - 1$, so by the requirement that $\kappa < \frac{p}{2} - 1$, we need that $\frac{3p}{4}(1-\frac{1}{q}) - 1 < \frac{p}{2} - 1$ in order to find such $\kappa$.
This is in turn equivalent with $q < 3$.
Finally, let us check the inequality $\frac{1+\kappa}{p} \le 2(1 - \beta) \le \frac{3}{2}(1-\frac{1}{q})$ by the choice of $\beta$.
This condition can be shown to be equivalent with $\kappa \le \frac{3p}{2}(1-\frac{1}{q}) - 1$.
Since $\kappa \ge 0$, we need that~$\frac{3p}{2}(1-\frac{1}{q}) - 1 \ge 0$, which is satisfied provided $\frac{2}{p} + \frac{3}{q} \le 3$.
The latter condition is already fulfilled for the present case of $p \ge 2$ and $q \ge 2$.
In summary, we have verified that in the present situation of~$q \in [2,3)$, there are $p \in [2,\infty)$, $\kappa \in [0,\frac{p}{2}-1) \cup \{0\}$ and $\beta \in (1-\frac{1+\kappa}{p},1)$ with $\frac{1+\kappa}{p} \le 2(1-\beta)$.

We now repeat the procedure for $q \ge 3$.
By \autoref{lem:nonlin est}(b), we then consider $\beta \in [\frac{1}{2},1)$, so $\beta < 1$ is already satisfied.
Moreover, we need that $\beta > 1 - \frac{1+\kappa}{p}$, which can be guaranteed for~$\frac{1+\kappa}{p} > 1 - \beta$, or $\kappa > p(1-\beta) - 1$.
By the requirement that $\kappa < \frac{p}{2} - 1$, this means that $p(1-\beta) - 1 < \frac{p}{2} - 1$, or, equivalently, $\beta > \frac{1}{2}$.
It remains to check that $\frac{1+\kappa}{p} \le 2(1-\beta)$, so $\kappa \le 2 p(1-\beta) - 1$.
Therefore, by $\kappa \ge 0$, we need that $2 p (1-\beta) \ge 0$, but this is already implied by $\beta \in (\frac{1}{2},1)$ and $p \ge 2$.

In total, the assertion of \autoref{thm:loc strong wp stoch fsi} follows from an application of \autoref{prop:loc ex result bilin SPDE}.
\end{proof}

We remark that the previous proof especially covers the local strong well-posedness result \autoref{thm:loc strong wp stoch fsi Hilbert space} in the Hilbert space case.
Thus, we proceed with the proof of the blow-up criterion \autoref{thm:blow-up crit stoch fsi}.

\begin{proof}[Proof of \autoref{thm:blow-up crit stoch fsi}]
From the above proof of \autoref{thm:loc strong wp stoch fsi}, it follows that $\beta = \frac{1}{4}(1+\frac{3}{q})$ and the associated choice $\frac{1+\kappa}{p} = 2(1-\beta) = \frac{3}{2} - \frac{3}{2q}$ corresponds to the critical case.
Note that this is only possible if $q \in [2,3)$, and we obtain $\rX_{1-\frac{1+\kappa}{p},p} = \rX_{\frac{3}{2q} - \frac{1}{2},p}$ and $\rX_{1-\frac{\kappa}{p}} = \rX_{\frac{1}{p}+\frac{3}{2q}-\frac{1}{2}}$.
By \autoref{prop:blow-up crit bilin SPDE}(a), the first part of the assertion is implied.
As a by-product of the proof of \autoref{thm:loc strong wp stoch fsi}, we also find that the case~$q \ge 3$ is always subcritical, so \autoref{prop:blow-up crit bilin SPDE}(b) yields the assertion of~(b).
\end{proof}

Finally, \autoref{thm:blow-up crit stoch fsi Hilbert space} is a consequence of \autoref{prop:blow-up crit bilin SPDE}(b) upon noting that the situation of $p = q = 2$ and $\kappa = 0$ is subcritical.
Indeed, the proof of \autoref{thm:loc strong wp stoch fsi} reveals that $\beta = \frac{5}{8}$ is the critical value, i.e., $2 (1 - \beta) = \frac{3}{4}$.
However, we have $\frac{1+\kappa}{p} = \frac{1}{2} <  \frac{3}{4} = 2 (1 - \beta)$, demonstrating the subcriticality.

\section{Concluding remarks and further discussion}\label{sec:concl rems & further discussion}

This section is devoted to the exposition of possible extensions of the results obtained in the previous sections, and we provide an overview of potential future works.

First, let us stress that further interesting functional analytic properties such as the $\cR$-boundedness of the $\Hinfty$-calculus of $-\Afs$ or the boundedness of the imaginary powers of $-\Afs$ could be deduced from \autoref{thm:bdd H00-calculus of fluid-structure op incompr Newtonian}.
However, we do not elaborate more on that and refer instead to \cite{DHP:03} or \cite[Section~4.5]{PS:16} for more details on these properties as well as their relation to the concepts studied here.

For completeness, let us also briefly elaborate on the case of a bounded fluid domain $\cF_0 \subset \R^3$.
In this case, we assume homogeneous Dirichlet boundary conditions for the fluid velocity on the outer fluid boundary.
The according lifting procedure is similar to \autoref{sec:bdd Hoo-calculus fluid-structure operator} and has been described in detail in \cite[Section~3]{MT:18}.
There are two differences with \autoref{sec:bdd Hoo-calculus fluid-structure operator}:
On the one hand, a refined spectral analysis reveals that the shift, which is indispensable in the exterior domain case, can be omitted. 
On the other hand, due to the boundedness of the underlying domain, the study of the Neumann problem for the pressure is more straightforward.
In fact, it is possible to follow \cite[Section~9]{FMM:98} and to employ \cite[Theorem~9.2]{FMM:98} to establish an analogue of \autoref{lem:lower reg Neumann problem}, but without restriction on $q \in (1,\infty)$.

Mimicking the procedure from \autoref{sec:bdd Hoo-calculus fluid-structure operator}, we introduce the analogue of the fluid-structure operator for~$\cF_0$ bounded, still denoted by $\cA_\fs$ by a slight abuse of notation, and we get a similar result as \autoref{thm:bdd H00-calculus of fluid-structure op incompr Newtonian}.
The ground space in this setting is as in \eqref{eq:ground space fluid-structure op}, so
\begin{equation*}
    \rX_0 = \{(u,\ell,\omega) \in \rL^q(\cF_0)^3 \times \C^3 \times \C^3 : u - D_{\lambda_0,u}(\ell,\omega) \in \rL_\sigma^q(\cF_0)\}.
\end{equation*}
Additionally invoking that $\C_+ \subset \rho(\cA_\fs)$ in this case, see \cite[Theorem~4.2]{MT:18}, we find that we may choose $\lambda = 0$ here.
This paves the way for the following result.
Note that the assertions in (a)--(c) can be derived from the boundedness of the $\Hinfty$-calculus as sketched before \autoref{cor:cons of H00-calc of fs op incompr Newtonian}.

\begin{thm}\label{thm:props fluid-structure op bdd dom}
Let $q \in (1,\infty)$.
Then for the fluid-structure operator $\cA_\fs$ on the bounded domain $\cF_0$, it holds that $-\cA_\fs \in \Hinfty(\rX_0)$ with $\phi_{-\cA_\fs}^\infty = 0$, and
\begin{enumerate}[(a)]
    \item the operator $-\cA_\fs$ has maximal $\rL^p$-regularity on $\rX_0$, and $\cA_\fs$ generates a bounded analytic and exponentially stable semigroup $(\mre^{t \cA_\fs})_{t \ge 0}$ of angle $\nicefrac{\pi}{2}$ on $\rX_0$,
    \item for $\alpha \in (0,1)$ and $\rX_\alpha$ related to $-\cA_\fs$ as specified in \autoref{lem:rel of Hinfty with other concepts}, we have $\rX_\alpha \cong [\rX_0,\rX_{\cA_\fs}]_\alpha$, and
    \item for $p$, $q \in [2,\infty)$ and $\kappa \in [0,\nicefrac{p}{2}-1)$, we get $(-\Afs,0) \in \cSMR_{p,\kappa}^\bullet$.
\end{enumerate}
\end{thm}

A similar result is also available in the 2D case, and we stress again that we allow for a rigid body of arbitrary shape with $\rC^3$-boundary.
To the best of our knowledge, \autoref{thm:bdd H00-calculus of fluid-structure op incompr Newtonian} and \autoref{thm:props fluid-structure op bdd dom} are the first results on the bounded $\Hinfty$-calculus of fluid-structure operators.
In the case of a fluid-filled rigid body, Mazzone et al.\ \cite{MPS:19a, MPS:19b} established the boundedness of the $\Hinfty$-calculus of the underlying linear operator.
However, the situation of a rigid body with a fluid-filled cavity is significantly different from the present fluid-rigid body interaction problem, and the domain of the associated operator is especially {\em not} non-diagonal.
A natural question is if similar results as in \autoref{sec:main results} can be obtained in the case of a rigid body of arbitrary shape, or of a bounded domain.
A difficulty in the former case is that the transformation to the fixed domain is more complicated.
In fact, common transformations to a fixed domain in this case either lead to the introduction of terms that violate the parabolicity of the problem, or terms that are non-local in time enter the equation.
For a bounded domain, another significant difficulty is the possibility of a collision of the rigid body with the outer fluid boundary.
It seems that the framework provided in \autoref{ssec:stoch max reg and semilin SPDEs} is insufficient for the analysis, and we leave the investigation to future study.

\medskip 

{\bf Acknowledgments.}
This research is supported by the Basque Government through the BERC 2022-2025 program and by the Spanish State Research Agency through BCAM Severo Ochoa excellence accreditation CEX2021-01142-S funded by MICIU/AEI/10.13039/501100011033 and through Grant PID2023-146764NB-I00 funded by MICIU/AEI/10.13039/501100011033 and cofunded by the European Union. 
Felix Brandt would like to thank the German National Academy of Sciences Leopoldina for support through the Leopoldina Fellowship Program with grant number~LPDS 2024-07. 
Arnab Roy acknowledges the support by  Grant RYC2022-036183-I funded by MICIU/AEI/10.13039/501100011033 and by ESF+.

The authors would also like to thank Antonio Agresti for helpful comments and fruitful discussions.

\bibliography{linfsi}
\bibliographystyle{siam}

\end{document}